\documentclass{amsart}

\usepackage[latin1]{inputenc}
\usepackage[english]{babel}
\usepackage{indentfirst}
\usepackage{amssymb}
\usepackage{amsthm}
\usepackage{xcolor}
\usepackage[all]{xy}
\usepackage[mathscr]{eucal}
\usepackage{hyperref}
\newcommand{\impli}{\Rightarrow}
\newcommand{\Nat}{\mathbb{N}}
\newcommand{\N}{\mathbb{N}}
\newcommand{\erre}{\mathbb{R}}
\newcommand{\sub}{\subseteq}

\newcommand{\tenp}{\hat\otimes_\pi}
\newcommand{\WDP}{\mathcal{WDP}}

\def\epsilon{\varepsilon}
\newtheorem{theo}{Theorem}[section]
\newtheorem{lem}[theo]{Lemma}%[section]
\newtheorem{pro}[theo]{Proposition}%[section]
\newtheorem{cor}[theo]{Corollary}%[theo]
\newtheorem{defi}[theo]{Definition}%[section]
\newtheorem{rem}[theo]{Remark}
\newtheorem{exa}[theo]{Example}

\newtheorem{theorem}[theo]{Theorem}
\newtheorem{proposition}[theo]{Proposition}

\theoremstyle{definition}

\newcommand{\iten}{\ensuremath{\widehat{\otimes}_\varepsilon}}
\newcommand{\pten}{\ensuremath{\widehat{\otimes}_\pi}}

\numberwithin{equation}{section}

\title{Topological properties in tensor products of Banach spaces}

\author[A. Avil\'es]{Antonio Avil\'es}
\address{Universidad de Murcia, Departamento de Matem\'{a}ticas, Facultad de Matem\'{a}ticas, 30100 Espinardo (Murcia), Spain.}
\email{avileslo@um.es}

\author[G. Mart\'inez-Cervantes]{Gonzalo Mart\'inez-Cervantes}
\address{Universidad de Alicante, Departamento de Matem\'aticas, Facultad de Ciencias, 03080 Alicante, Spain.}
\email{gonzalo.martinez@ua.es}

\author[J. Rodr\'iguez]{Jos\'e Rodr\'{i}guez}
\address{Departamento de Ingenier\'{i}a y Tecnolog\'{i}a de Computadores,
Facultad de Inform\'{a}tica, Universidad de Murcia, 30100 Espinardo (Murcia), Spain.}  
\email{joserr@um.es}

\author[A. Rueda Zoca]{Abraham Rueda Zoca}
\address{Universidad de Murcia, Departamento de Matem\'{a}ticas, Facultad de Matem\'{a}ticas, 30100 Espinardo (Murcia), Spain.}
\email{abraham.rueda@um.es}



\thanks{The research is partially supported by grants MTM2017-86182-P 
(funded by MCIN/AEI/10.13039/501100011033 and ``ERDF A way of making Europe'') and 
20797/PI/18 (funded by {\em Fundaci\'on S\'eneca}).
A. Rueda Zoca was also partially supported by: grant PGC2018-093794-B-I00 
(funded by MCIN/AEI/10.13039/501100011033 and ``ERDF A way of making Europe''), 
grant FJC2019-039973 (funded by MCIN/AEI/10.13039/501100011033) and grants A-FQM-484-UGR18 and FQM-0185 (funded by {\em Junta de Andaluc\'{i}a}).}

\begin{document}

\begin{abstract}
Given two Banach spaces~$X$ and~$Y$, we analyze when the projective tensor product $X\pten Y$ 
has Corson's property~(C) or is weakly Lindelöf determined (WLD), subspace of a weakly compactly generated (WCG) space
or subspace of a Hilbert generated space. For instance, we show that: (i)~$X\pten Y$ is WLD if and only if both $X$ and $Y$ are WLD
and all operators from $X$ to~$Y^*$ and from $Y$ to~$X^*$ have separable range; (ii) $X\pten Y$ is subspace of a WCG space
if the same holds for both $X$ and~$Y$ under the assumption that every operator from $X$ to~$Y^*$ is compact;
(iii) $\ell_p(\Gamma)\pten \ell_q(\Gamma)$ is subspace of a Hilbert generated space for any $1< p,q<\infty$ such that $1/p+1/q<1$
and for any infinite set~$\Gamma$.
We also pay attention to the injective tensor product $X\iten Y$. In this case, the stability of property~(C) and 
the property of being WLD turn out to be closely related to the condition that all regular Borel probability measures on the dual ball have countable Maharam type. 
Along this way, we generalize a result of Plebanek and Sobota that if $K$ is a compact space such that $C(K\times K)$ has property~(C), then 
all regular Borel probability measures on~$K$ have countable Maharam type. This generalization provides a consistent negative answer to a question of
Ruess and Werner about the preservation of the $w^*$-angelicity of the dual unit ball under injective tensor products.   
\end{abstract}

\maketitle

\section{Introduction}\label{section:Intro}

The projective tensor product $\ell_2 \pten \ell_2$ is not reflexive, because it contains an isometric copy of~$\ell_1$; in fact, 
a copy is spanned by the sequence $(e_n\otimes e_n)_{n\in \N}$, where $(e_n)_{n\in \N}$ is the usual basis of~$\ell_2$ (see, e.g., \cite[Example~2.10]{rya}). 
More generally, given two reflexive Banach spaces $X$ and~$Y$, their projective tensor product
$X \pten Y$ is reflexive whenever every operator from $X$ to~$Y^*$ is compact, and the converse holds provided $X$ or $Y$ has the approximation property
(see, e.g., \cite[Theorems~4.19 and~4.21]{rya}).
This fact and Pitt's theorem imply that, given $1<p,q<\infty$ and a non-empty index set~$\Gamma$, 
the space $\ell_p(\Gamma)\pten \ell_q(\Gamma)$ is reflexive if and only if $1/p+1/q < 1$. 
Actually, the argument for $\ell_2 \pten \ell_2$ can be adapted to deduce 
that $\ell_p(\Gamma)\pten \ell_q(\Gamma)$ contains an isometric copy of~$\ell_1(\Gamma)$ whenever
$1/p+1/q \geq 1$ (see Proposition~\ref{pro:embedding-lp-lq}), which for 
uncountable~$\Gamma$ implies that $\ell_p(\Gamma)\pten \ell_q(\Gamma)$ even fails other Banach space properties,
much weaker than being reflexive, that have been thoroughly studied over the years, like being weakly compacty generated (WCG),
weakly Lindelöf determined (WLD) or having Corson's property~(C).

The point is that it is difficult to handle weak compactness in projective tensor products. The following result goes back to
\cite[Theorem~16]{die8}:

\begin{theo}\label{theo:Diestel}
Let $X$ and $Y$ be Banach spaces such that either $X$ or $Y$ has the Dunford-Pettis property. Then
$W_X\otimes W_Y$ is relatively weakly compact in $X\pten Y$ whenever $W_X \sub X$ and $W_Y \sub Y$ are 
relatively weakly compact. Consequently, $X\pten Y$ is WCG whenever $X$ and $Y$ are WCG.
\end{theo}

A more involved result by Talagrand (see \cite[Th\'{e}or\`{e}me 5.1(v)]{tal1}) states that 
if $X$ and $Y$ are weakly $\mathcal{K}$-analytic (resp., weakly $\mathcal{K}$-countably determined) Banach spaces such that either $X$ or $Y$ has the Dunford-Pettis property, then $X \tenp Y$ is weakly $\mathcal{K}$-analytic (resp., weakly $\mathcal{K}$-countably determined).

It is also natural to consider such type of questions for the injective tensor product $X\iten Y$ of two Banach spaces~$X$ and~$Y$.
While reflexivity is not preserved in general (for instance, its is easy to check that
$\ell_2\iten \ell_2$ contains an isometric copy of~$c_0$; cf. \cite[Theorem 16.73]{fab-ultimo}),
the analogue of Theorem~\ref{theo:Diestel} for injective tensor products is valid for arbitrary Banach spaces
(see \cite[Theorem~2.1]{rue-wer}), and the same can be said about Talagrand's results above (see \cite[Th\'{e}or\`{e}me 5.1(iv)]{tal1}).  
It is also worth mentioning a result of Pol saying that the space $C(K,Y)=C(K)\iten Y$ has property~(C)
whenever $K$ is Eberlein compact and $Y$ has property~(C) (see \cite[Section~4]{pol}).

In this paper we study several topological properties for the projective and injective tensor products of Banach spaces. Namely, we focus on
property~(C) and the following classes of Banach spaces: WLD spaces, subspaces of WCG spaces and subspaces of Hilbert generated spaces.  
The paper is organized as follows.

In Section~\ref{section:Preliminaries} we fix the terminology and include some preliminaries on spaces of operators, tensor products and Banach spaces.

In Section~\ref{section:C} we discuss the impact of property~(C) in projective tensor products. It turns out that
if $X$ and $Y$ are Banach spaces such that $X\tenp Y$ has property~(C)
and $X$ has the bounded approximation property or the separable complementation property, then every operator from~$X$ to~$Y^*$ 
has $w^*$-separable range (Corollary~\ref{cor:SCPcct}). As an application, we get Kalton's
result~\cite{kal74} that $\mathcal{L}(X)$ cannot be reflexive unless the Banach space $X$ is separable (Corollary~\ref{cor:Kalton}).

In Section~\ref{section:WLDpten} we analyze the property of being WLD in projective tensor products. 
A complete characterization is obtained, namely: given two Banach spaces~$X$ and~$Y$, the space
$X\pten Y$ is WLD if and only if $X$ and $Y$ are WLD and every operator from~$X$ to~$Y^*$ and from~$Y$ to~$X^*$
has (norm) separable range (Theorem~\ref{theorem:caractWLD}).
This allows to elucidate when Lebesgue-Bochner spaces $L_1(\mu,Y)=L_1(\mu)\pten Y$ are WLD (Corollary~\ref{cor:Lebesgue-Bochner}).

In Section~\ref{section:WCGpten} we consider the property of being subspace of a WCG space
in projective tensor products. In the spirit of Theorem~\ref{theo:Diestel}, we prove that  
if $X$ and $Y$ are Banach spaces such that $X$ is subspace of a WCG space, $Y$ is WCG and
either $X$ has the dual quantitative Dunford-Pettis property or $Y$ has the direct
quantitative Dunford-Pettis property of Kacena, Kalenda and Spurn\'{y}~\cite{kac-alt}
(both properties are fulfilled by all $\mathcal{L}_1$ spaces and all $\mathcal{L}_\infty$ spaces), then 
$X \pten Y$ is subspace of a WCG space (Theorem~\ref{theo:pten-subspaceWCG-qDP}).
The same conclusion holds if $X$ and $Y$ are subspaces of WCG spaces and every operator from~$X$ to~$Y^*$ is compact
(Corollary~\ref{cor:pten-subspaceWCG-qDPalloperators}).

In Section~\ref{section:Hilbertpten} we pay attention to the property of being subspace of a Hilbert generated space
in projective tensor products. We prove that, for any non-empty index set~$\Gamma$, 
the spaces $c_0(\Gamma)\pten c_0(\Gamma)$, $c_0(\Gamma) \pten \ell_q(\Gamma)$ (for any $1<q<\infty$) and 
$\ell_p(\Gamma)\pten \ell_q(\Gamma)$ for any $1<p,q<\infty$ with $1/p+1/q<1$ are subspaces of Hilbert generated spaces 
(Theorem~\ref{theo:lp-pten-lq-sHG} and Corollary~\ref{cor:c0-pten-c0-sHG}).

In Section~\ref{section:WLDiten} we address similar questions for injective tensor products. Some properties like being WCG, 
Hilbert generated or subspace of such spaces are easily seen to be stable under injective tensor products, hence
we focus on WLD spaces and property~(C). Given two Banach spaces~$X$ and~$Y$, the injective tensor product
$X\iten Y$ is shown to be WLD if and only if $X$ and $Y$ are WLD and every integral operator from~$X$ to~$Y^*$
and from~$Y$ to~$X^*$ has (norm) separable range (Theorem~\ref{theo:caraWLDinject}).
This happens if $X$ and $Y$ are WLD and either $(B_{X^*},w^*)$ or $(B_{Y^*},w^*)$ has property~(M)
(i.e., every regular Borel probability measure on it has separable support), see Corollary~\ref{cor:WLDitenM}. 
A Banach space~$X$ is WLD and $(B_{X^*},w^*)$ has property~(M) if and only if $X\iten X$ is WLD (Corollary~\ref{cor:WLDiten-noM}).
As to property~(C), we prove that if a Banach space $X$ has the bounded approximation property or the separable complementation property
and $X\iten X$ has property~(C), then every regular Borel probability measure on $(B_{X^*},w^*)$ has countable Maharam type
(Corollary~\ref{cor:PropertyCiten-MS}). This generalizes a result of Plebanek and Sobota~\cite{ple-sob}
who proved the same statement when $X=C(K)$ for some compact space~$K$. It also provides
a consistent negative answer to a question of Ruess and Werner~\cite{rue-wer} about the 
preservation of the $w^*$-angelicity of the dual unit ball under injective tensor products
(Remark~\ref{rem:CH}).  

Finally, in Section~\ref{section:Questions} we collect several open questions related to our work.

\section{Preliminaries}\label{section:Preliminaries}

Our topological spaces are assumed to be Hausdorff and our locally convex spaces and Banach spaces are assumed to be
over the real field. The cardinality of a set $\Gamma$ is denoted by~$|\Gamma|$ and the symbol
$\omega_1$ stands for the first uncountable ordinal. By a {\em compact space} we mean a compact topological space. 
Given $1<p<\infty$ we denote by $p^*$ its Hölder conjugate, i.e., $1/p+1/p^*=1$.
Given a subset $D$ of a locally convex space~$E$, the linear subspace of~$E$ generated by~$D$ is denoted by 
${\rm span}(D)$ and its closure by $\overline{{\rm span}}(D)$. We write 
${\rm co}(D)$ (resp., $\overline{{\rm co}}(D)$) to denote the convex hull (resp., closed convex hull) of~$D$.
By a {\em subspace} of a Banach space we mean a norm closed linear subspace.
The topological dual of a Banach space~$X$ is denoted by~$X^*$ and
we write $w^*$ (resp., $w$) to denote the weak$^*$-topology (resp., weak topology) on~$X^*$ (resp., $X$). 
The evaluation of $x^*\in X^*$ at $x\in X$ is denoted by either $x^*(x)$ or $\langle x^*,x\rangle$.
We write $B_X=\{x\in X:\|x\|\leq 1\}$ to denote the closed unit ball of~$X$. 
A {\em Markushevich basis} in~$X$ is a biorthogonal system $\{(x_i,x_i^*):i\in I\}\sub X\times X^*$ 
such that $X=\overline{{\rm span}}(\{x_i:i\in I\})$ and $\{x_i^*:i\in I\}$ separates the points of~$X$. 
Given two sets $C_1,C_2 \sub X$, its Minkowski sum is $C_1+C_2:=\{x_1+x_2:\, x_1\in C_1, \, x_2\in C_2\}$.

All unexplained terminology can be found in standard references
like~\cite{fab-ultimo} and \cite{fab-alt-JJ} (Banach spaces) and~\cite{rya} (tensor products). 
The survey paper \cite{ziz} is a good source of information on non-separable Banach spaces.

\subsection{Spaces of operators}\label{subsection:Operators}

Given two Banach spaces $X$ and $Z$, we write $\mathcal{L}(X,Z)$
to denote the Banach space of all operators (i.e., linear and continuous maps) from~$X$ to~$Z$,
equipped with the operator norm. The space $\mathcal{L}(X,Z)$ can be equipped with several locally 
convex topologies weaker than the norm topology. The {\em strong operator topology (SOT)}  
on $\mathcal{L}(X,Z)$ is the one for which the sets 
$$
	\{T\in \mathcal{L}(X,Z): \, \|T(x)\|<\epsilon\} 
	\quad \text{where $x\in X$ and $\epsilon>0$}
$$ 
are a subbasis of open neighborhoods of~$0$. Therefore, a net $(T_\alpha)$ in $\mathcal{L}(X,Z)$ is SOT-convergent 
to~$T\in \mathcal{L}(X,Z)$ if and only if $T_\alpha(x)\to T(x)$ in norm for every $x\in X$. If in addition 
$Z=Y^*$ for some Banach space~$Y$, then the {\em weak$^*$ operator topology (W$^*$OT)} 
on $\mathcal{L}(X,Y^*)$ is the locally convex topology for which the sets 
$$
	\{T\in \mathcal{L}(X,Y^*): \, |\langle T(x),y\rangle|<\epsilon\}
	\quad\text{where $x\in X$, $y\in Y$ and $\epsilon>0$}
$$ 
are a subbasis of open neighborhoods of~$0$. Therefore, in this case a net $(T_\alpha)$ in $\mathcal{L}(X,Y^*)$ is W$^*$OT-convergent 
to~$T\in \mathcal{L}(X,Y^*)$ if and only if $T_\alpha(x)\to T(x)$ in the weak$^*$-topology for every $x\in X$. 

We will consider the following subspaces of $\mathcal{L}(X,Z)$: 
\begin{eqnarray*}
	\mathcal{K}(X,Z) & = & \{T\in \mathcal{L}(X,Z): \, \text{$T$ is compact}\}, \\
	\mathcal{W}(X,Z) & = & \{T\in \mathcal{L}(X,Z): \, \text{$T$ is weakly compact}\}, \\
	\mathcal{DP}(X,Z) & = & \{T\in \mathcal{L}(X,Z): \, \text{$T$ is Dunford-Pettis}\}, \\
	\mathcal{S}(X,Z) & = & \{T\in \mathcal{L}(X,Z): \, \text{$T$ has separable range}\}.
\end{eqnarray*}
As usual, we write $\mathcal{L}(X)$, $\mathcal{K}(X)$ and so on to denote $\mathcal{L}(X,X)$, $\mathcal{K}(X,X)$, etc.

\subsection{The projective tensor product}\label{subsection:pten}

Given two Banach spaces~$X$ and~$Y$, we denote by $\mathcal{B}(X,Y)$
the Banach space of all continuous bilinear maps $S:X\times Y \to \erre$, equipped with the norm
$\|S\|=\sup\{|S(x,y)|:\, x \in B_X,\, y \in B_Y\}$. Each element of~$\mathcal{B}(X,Y)$ induces
a linear functional (denoted in the same way) in the algebraic tensor product $X\otimes Y$.
The {\em projective tensor product} of~$X$ and~$Y$, denoted by $X\pten Y$, is the completion
of $X\otimes Y$ when equipped with the norm
$$
	\|u\|=\sup\{|S(u)|: \, S\in \mathcal{B}(X,Y), \, \|S\|\leq 1\}, \quad u \in X\otimes Y.
$$
Thus each $S \in \mathcal{B}(X,Y)$ defines an element of $(X\pten Y)^*$ and, in fact,
this correspondence is an isometric isomorphism from $\mathcal{B}(X,Y)$ onto~$(X \pten Y)^*$.
For each $S\in \mathcal{B}(X,Y)$ we define $S_X\in \mathcal{L}(X,Y^*)$ and $S_Y\in \mathcal{L}(Y,X^*)$ by
$$
	S_X(x)(y)=S_Y(y)(x):=S(x,y) \quad
	\text{for all $x\in X$ and $y\in Y$}.
$$
The map $S \mapsto S_X$ (resp., $S\mapsto S_Y$) is an isometric isomorphism from
$\mathcal{B}(X,Y)$ onto $\mathcal{L}(X,Y^*)$ (resp., $\mathcal{L}(Y,X^*)$). Under these identifications,
the weak$^*$-topology of $(X\pten Y)^*$ coincides with the W$^*$OT-topology on bounded subsets
of $\mathcal{L}(X,Y^*)$ (resp., $\mathcal{L}(Y,X^*)$).

\subsection{The injective tensor product}\label{subsection:iten}

Let $X$ and~$Y$ be Banach spaces. For each $x^*\in X^*$ and for each $y^*\in Y^*$
we have $x^*\otimes y^*\in \mathcal{B}(X,Y)$ defined by
$$
	(x^*\otimes y^*)(x,y):=x^*(x)y^*(y)
	\quad
	\text{for all $x\in X$ and $y\in Y$.}
$$
The {\em injective tensor product} of~$X$ and~$Y$, denoted by $X\iten Y$, is the completion
of $X\otimes Y$ when equipped with the norm
$$
	\|u\|=\sup\{|(x^*\otimes y^*)(u)|: \, x^*\in B_{X^*}, \, y^*\in B_{Y^*}\}, \quad u \in X\otimes Y.
$$
The identity map on $X\otimes Y$ can be extended to an operator from $X\pten Y$ to $X\iten Y$ with norm~$1$ and dense range.
Each element of $(X \iten Y)^*$ can be identified with some $S\in\mathcal{B}(X,Y)$ for which $S_X$ (equivalently,~$S_Y$) 
is {\em Pietsch integral}, i.e., it factors as  
$$
	\xymatrix@R=3pc@C=3pc{X
	\ar[r]^{S_X} \ar[d]_{U} & Y^*\\
	L_\infty(\mu)  \ar[r]^{I}  & L_1(\mu) \ar[u]_{V} \\
	}
$$
for some finite measure~$\mu$, where $I$ is the formal inclusion operator and $U$ and $V$ are operators.
The norm of~$S$ as an element of~$(X\iten Y)^*$ is the {\em Pietsch integral norm} $\|S_X\|_{\rm int}$ of~$S_X$, 
which is defined as the infimum of the quantities $\|U\|\|V\|\mu(\Omega)$ over all factorizations as above. Clearly, $\|S_X\|\leq \|S_X\|_{{\rm int}}$.

The linear subspace of $\mathcal{L}(X,Y^*)$ (resp., $\mathcal{L}(Y,X^*)$) consisting of all Pietsch integral operators
will be denoted by $\mathcal{I}(X,Y^*)$ (resp., $\mathcal{I}(Y,X^*)$). Under the identifications above,
the weak$^*$-topology of $(X\iten Y)^*$ coincides with the W$^*$OT-topology on $\|\cdot\|_{{\rm int}}$-bounded subsets
of $\mathcal{I}(X,Y^*)$ (resp., $\mathcal{I}(Y,X^*)$).

\subsection{Weakly compactly generated and Hilbert generated spaces}\label{subsection:WCG}
We refer the reader to \cite[Chapter~13]{fab-ultimo} and \cite[Sections~6.2 and~6.3]{fab-alt-JJ}
for complete information on these topics. A Banach space $X$ is said to be {\em weakly compactly generated (WCG)}
if there is a weakly compact set $G \sub X$ such that $X=\overline{{\rm span}}(G)$. 
Separable spaces and reflexive spaces are WCG. The class of WCG Banach spaces
is not closed under subspaces (this was first discovered by Rosenthal, see~\cite{ros-J-8}).
A Banach space $X$ is subspace of a WCG space if and only if $(B_{X^*},w^*)$ is Eberlein. 
Recall that a compact space~$K$ is said to be {\em Eberlein} if it is homeomorphic to a weakly compact subset of 
a Banach space or, equivalently, to a weakly compact subset of $c_0(\Gamma)$
for some non-empty set~$\Gamma$.

The Davis-Figiel-Johnson-Pe{\l}czy\'{n}ski factorization procedure
applies to deduce that a Banach space $X$ is WCG if and only if there exist a reflexive Banach space~$Y$
and an operator from~$Y$ to~$X$ with dense range. If~$Y$
can be chosen to be a Hilbert space, then $X$ is said to be {\em Hilbert generated}. The class of Hilbert generated
spaces includes all separable spaces and $L_1(\mu)$ for any finite measure~$\mu$. It is neither closed under subspaces,
as Rosenthal's aforementioned counterexample to the heredity problem for WCG spaces shows. 
A Banach space $X$ is subspace of a Hilbert generated space if and only if $(B_{X^*},w^*)$ is uniform Eberlein. 
Recall that a compact space~$K$
is said to be {\em uniform Eberlein} if it is homeomorphic to a weakly compact subset of a Hilbert space.
Every super-reflexive Banach space is subspace of a Hilbert generated space, but there are reflexive spaces
which are not.

\subsection{Weakly Lindelöf determined spaces}\label{subsection:WLD}

The reader is referred to \cite[Section~14.5]{fab-ultimo}, \cite[Sections~5.4 and~5.5]{fab-alt-JJ}
and \cite[Chapter~7]{fab-J} for complete information on this topic.
Given a non-empty set~$\Gamma$, the topology of pointwise convergence on $\mathbb{R}^\Gamma$
is denoted by $\tau_p(\Gamma)$. We denote by $\ell_\infty^c(\Gamma)$ the subspace of~$\ell_\infty(\Gamma)$ consisting
of all bounded functions $f:\Gamma \to \erre$ having countable support (i.e., the set $\{\gamma\in \Gamma:f(\gamma)\neq 0\}$ is countable)
and we write 
$$
	\Sigma([-1,1]^\Gamma):=[-1,1]^\Gamma \cap \ell_\infty^c(\Gamma)=B_{\ell_\infty^c(\Gamma)}.
$$
A compact space~$K$ is said to be {\em Corson} if it embeds homeomorphically into $(\Sigma([-1,1]^\Gamma),\tau_p(\Gamma))$ 
for some non-empty set~$\Gamma$. Every Eberlein compact space is Corson.
A Banach space~$X$ is said to be {\em weakly Lindelöf determined (WLD)} if there exist
a non-empty set~$\Gamma$ and an injective operator $\Phi: X^*\to \ell_\infty^c(\Gamma)$ which is
$w^*$-to-$\tau_p(\Gamma)$ continuous. This is equivalent to the fact that $(B_{X^*},w^*)$ is Corson.
The class of WLD spaces is closed under subspaces and is strictly larger than the class of subspaces 
of WCG spaces (see, e.g., \cite[Section~8.4]{fab-J}).
Every WLD Banach space~$X$ admits a Markushevich basis and, if $\{(x_i,x_i^*):i\in I\}\sub X\times X^*$
is any Markushevich basis in~$X$, then for each $x^*\in X^*$ the set $\{i\in I:x^*(x_i)\neq 0\}$ is countable.

\subsection{Corson's property~(C)}\label{subsection:C}

A Banach space~$X$ is said to have {\em Corson's property~(C)} if every family of convex closed subsets of~$X$ with empty intersection
has a countable subfamily with empty intersection. Pol~\cite{pol} showed that this is equivalent to the fact that $(B_{X^*},w^*)$ has convex countable tightness
(see, e.g., \cite[Theorem~14.37]{fab-ultimo}), in the following sense:

\begin{defi} 
A convex subset $C$ of a locally convex space $E$ is said to have {\em convex countable tightness} if for every convex set $D \sub C$ and for every 
$x\in \overline{D}$ there is a countable set $D_0\sub D$ such that $x\in \overline{D_0}$.
\end{defi}

Every WLD Banach space has property~(C) (see, e.g. \cite[Theorem~5.37]{fab-alt-JJ}), but the converse fails in general
(see, e.g., \cite[Theorem~14.39]{fab-ultimo}).

\subsection{Measures on compact spaces}\label{subsection:Measures}

Given a compact space~$K$, we denote by $P(K)$ the set of all regular Borel probability measures on~$K$.
Recall that a Corson compact space~$K$ is said to have {\em property~(M)} if every $\mu\in P(K)$ has separable support.
Since every separable subset of a Corson compact space is metrizable (see, e.g., \cite[Exercise~14.58]{fab-ultimo}), 
it follows that a Corson compact space~$K$
has property~(M) if and only if every $\mu\in P(K)$ has metrizable support, which in turn implies that
$\mu$ has {\em countable Maharam type} (i.e., the space $L_1(\mu)$ is separable). So, any Corson compact space having property~(M) belongs to
the following class:

\begin{defi}\label{defi:MS}
A compact space~$K$ is said to belong to the {\em class~MS} if every $\mu\in P(K)$ has countable Maharam type.
\end{defi}

Conversely, {\em for any WLD Banach space~$X$, the Corson compact space $(B_{X^*},w^*)$ 
has property~(M) if and only if it belongs to the class~MS.}
Indeed, for an arbitrary Banach space~$X$, any regular Borel probability measure on~$(B_{X^*},w^*)$ 
having countable Maharam type is concentrated on a $w^*$-separable set (see the remark after Theorem~B.2 in~\cite{avi-mar-ple})
and so it has $w^*$-metrizable support whenever $X$ is WLD (cf. \cite[Theorem~2.2]{ple2}).

A Corson compact space~$K$ has property~(M) if and only if the space $C(K)$ is WLD 
(see, e.g., \cite[Theorem~5.57]{fab-alt-JJ}). Every Eberlein compact space has property~(M). However,
the question of whether every Corson compact space has property~(M) is undecidable in ZFC. 
On the one hand, under MA+$\neg$CH every Corson compact space has property~(M) (see, e.g., \cite[Theorem~5.62]{fab-alt-JJ}). 
On the other hand, under CH there exist Corson compact spaces without property~(M); 
in fact, there exist WLD Banach spaces~$X$ for which $(B_{X^*},w^*)$ fails property~(M) (see \cite[Corollary 4.4]{ple2}).
We also stress that under CH there are Corson compact spaces belonging to the class MS which fail property~(M)
(see \cite{ple6} and the references therein).

\section{Property~(C) and projective tensor products}\label{section:C}

We begin this section by pointing out that the question of whether the projective tensor product preserves property~(C)
has a simple answer when one of the spaces is separable:

\begin{pro}\label{pro:separable-pten-C}
Let $X$ and $Y$ be Banach spaces. If $X$ is separable and $Y$ has property~(C), then $X\pten Y$ has property~(C).
\end{pro}

The proof relies on the fact that
convex countable tightness is preserved by countable products of compact convex sets, see Lemma~\ref{lem:CCTproduct-countable} below.
The later can be proved by essentially the same argument that the product of countably many compact spaces having countable tightness 
also has countable tightness (see, e.g., \cite[p.~112, 5.9]{juh}). We include a detailed proof for the convenience of the reader.

\begin{defi}
A subset $D$ of a topological space is said to be {\em $\omega$-closed} if $\overline{D_0}\sub D$ for every
countable set $D_0\sub D$.
\end{defi}

\begin{lem}\label{lem:CCT}
Let $E$ be a locally convex space and $C \sub E$ be a closed convex set. Then $C$ has convex countable tightness if and only if every
$\omega$-closed convex subset of~$C$ is closed.
\end{lem}
\begin{proof} The `only if' part is immediate. Suppose now that every $\omega$-closed convex
subset of~$C$ is closed and take any convex set $D \sub C$. Define
$$
	D_1:=\bigcup \big\{\overline{U}: \, \text{$U \sub D$ is countable}\big\}=\bigcup \big\{\overline{{\rm co}}(U): \, \text{$U \sub D$ is countable}\big\},
$$ 
so that $D \sub D_1 \sub \overline{D} \sub C$. Clearly, $D_1$ is $\omega$-closed and convex, hence $D_1$ is closed
and therefore $\overline{D}=D_1$, as required.
\end{proof}

\begin{lem}\label{lem:CCTproduct}
Let $E_1$ and $E_2$ be locally convex spaces and let $C_1 \sub E_1$ and $C_2 \sub E_2$ be compact convex sets 
having convex countable tightness. Then $C_1 \times C_2$ has convex countable tightness in~$E_1\times E_2$.
\end{lem}
\begin{proof}
By Lemma~\ref{lem:CCT}, it suffices to show that every $\omega$-closed convex set $H \sub C_1\times C_2$ is closed.
Take $(x_1,x_2)\in \overline{H}$ and let us prove that $(x_1,x_2)\in H$. Since the map
$$
	\phi: E_2 \to \{x_1\}\times E_2,
	\quad
	\phi(y):=(x_1,y),
$$
is an affine homeomorphism, $\phi(C_2)=\{x_1\}\times C_2$ is a convex compact set having convex countable tightness. Since its convex subset
$$
	H_0:=H \cap (\{x_1\}\times C_2)
$$
is $\omega$-closed, it follows that $H_0$ is closed, hence compact. 

Let $\pi_2: C_1\times C_2 \to C_2$ be the canonical projection. To finish the proof we will check that $x_2\in \pi_2(H_0)$.
By contradiction, suppose that $x_2\not\in \pi_2(H_0)$. Since $\pi_2(H_0)$ is compact, 
there is a closed convex neighborhood $V$ of~$x_2$ in~$C_2$ such that $V \cap \pi_2(H_0)=\emptyset$. Since 
$W:=C_1\times V$ is a neighborhood of $(x_1,x_2)$ in~$C_1\times C_2$ 
and $(x_1,x_2)\in \overline{H}$, we have
$(x_1,x_2)\in \overline{H \cap W}$. 

Let $\pi_1: C_1\times C_2 \to C_1$ be the canonical projection. 
Since $\pi_1$ is continuous and $C_1\times C_2$ is compact, $\pi_1$ is a closed map (i.e., $\pi_1(F)$ is closed whenever $F \sub C_1 \times C_2$ is closed). 
Since $H \cap W \sub C_1\times C_2$ is $\omega$-closed, it follows that $\pi_1(H\cap W)$ is $\omega$-closed as well.  
Bearing in mind that $\pi_1(H\cap W)$ is convex and that $C_1$ has convex countable tightness, we conclude that $\pi_1(H\cap W)$ is closed.
Now, the continuity of $\pi_1$ implies that $x_1\in \pi_1(H\cap W)$, so there is $y\in C_2$ such that $(x_1,y)\in H \cap W$. In particular, $(x_1,y)\in H_0$ and $y\in V$,
which contradicts the fact that $V\cap \pi_2(H_0)=\emptyset$.  
\end{proof}

\begin{lem}\label{lem:CCTproduct-countable}
Let $(E_n)_{n\in \N}$ be a sequence of locally convex spaces and, for each $n\in \N$, let $C_n \sub E_n$ be a compact convex set 
having convex countable tightness. Then $\prod_{n\in \N}C_n$ has convex countable tightness in~$\prod_{n\in \N}E_n$.
\end{lem}
\begin{proof}
We will check that every $\omega$-closed convex set $H \sub \prod_{n\in \N}C_n$ is closed (and then Lemma~\ref{lem:CCT} applies).
For each finite set $F \sub \N$, let 
$$
	\pi_F: \prod_{n\in \N}C_n \to \prod_{n\in F} C_n
$$ 
be the canonical projection. Since $\pi_F$ is continuous and $\prod_{n\in \N}C_n$ is compact, $\pi_F$~is a closed map. This fact and the $\omega$-closedness of~$H$
imply that the set $\pi_F(H)$ is $\omega$-closed as well. Since $H$ is convex, so is $\pi_F(H)$. From the fact that $\prod_{n\in F}C_n$ has convex countable tightness
(which follows by induction from Lemma~\ref{lem:CCTproduct}), we conclude that $\pi_F(H)$ is closed.

Fix $x\in \overline{H}$. For each finite set $F \sub \N$ the map~$\pi_F$ is continuous, therefore $\pi_F(x)\in \pi_F(H)$
and we choose $x_F\in H$ such that $\pi_F(x)=\pi_F(x_F)$. Clearly, $D:=\{x_F: \, \text{$F \sub \N$ is finite}\}$ is 
a countable subset of~$H$ with $x\in \overline{D}$. 
Since $H$ is $\omega$-closed, we have $\overline{D}\sub H$ and so $x\in H$. This shows that $H$ is closed.
\end{proof}

\begin{proof}[Proof of Proposition~\ref{pro:separable-pten-C}] 
Let $(x_n)_{n\in\N}$ be a dense sequence in~$B_X$. Then
the map
$$
	\xi: (B_{(X\pten Y)^*},w^*) \to (B_{Y^*},w^*)^{\N},
	\quad
	\xi(S):=(S_X(x_n))_{n\in \N},
$$
is an affine homeomorphic embedding. Since $(B_{Y^*},w^*)^{\N}$ has convex countable tightness
(by Lemma~\ref{lem:CCTproduct-countable}), the same holds
for $(B_{(X\pten Y)^*},w^*)$. 
\end{proof}

As we already mentioned, the projective tensor product of two Banach spaces having property~(C) can fail property~(C).
An example is $\ell_p(\Gamma) \pten \ell_q(\Delta)$ for uncountable sets $\Gamma$ and $\Delta$
and $1<p,q<\infty$ satisfying $1/p+1/q\geq 1$. The point is that, under such assumptions, $\ell_p(\Gamma) \pten \ell_q(\Delta)$
contains a subspace isometric to~$\ell_1(\omega_1)$, which fails property~(C) (and this property is inherited by subspaces). 
While that embedding
might be known for specialists, we include a proof for the sake of completeness.

\begin{pro}\label{pro:embedding-lp-lq}
Let $\Gamma$ and $\Delta$ be uncountable sets and let $1<p,q<\infty$ be such that $1/p+1/q \geq 1$.
Then $\ell_p(\Gamma)\pten \ell_q(\Delta)$ contains a subspace isometric to~$\ell_1(\kappa)$, where
$\kappa=\min\{|\Gamma|,|\Delta|\}$.
\end{pro}
\begin{proof}
We can assume with no loss of generality that $\Gamma=\Delta$, since $\ell_q(\Gamma)$ is a $1$-complemented subspace of $\ell_q(\Delta)$ 
if $|\Gamma|\leq |\Delta|$ (and then \cite[Proposition~2.4]{rya} applies). Let $\{e_i:i\in \Gamma\}$ and $\{\tilde{e_i}:i\in \Gamma\}$
be the canonical bases of $\ell_p(\Gamma)$ and $\ell_q(\Gamma)$, respectively.
Let us prove that the set $\{e_i\otimes \tilde{e_i}:i\in \Gamma\} \sub \ell_p(\Gamma)\pten \ell_q(\Gamma)$ 
is isometrically equivalent to the canonical basis of~$\ell_1(\Gamma)$. To this end, pick a finite set $F\subseteq \Gamma$, take 
$\lambda_i\in \mathbb R$ for every $i\in F$, and let us prove that
\begin{equation}\label{eqn:lplq}
	\left\Vert \sum_{i\in F}\lambda_i e_i\otimes \tilde{e_i}\right\Vert=\sum_{i\in F}\vert \lambda_i\vert.
\end{equation}
The inequality ``$\leq$'' is obvious.
Define a bilinear functional $S:\ell_p(\Gamma)\times \ell_q(\Gamma)\to \mathbb R$ by
$$
	S(x,y):=\sum_{i\in F}{\rm sign}(\lambda_i)x_iy_i
	\quad
	\text{for all $x=(x_i)_{i\in \Gamma}\in \ell_p(\Gamma)$ and $y=(y_i)_{i\in \Gamma}\in \ell_q(\Gamma)$.}
$$
Let us prove that $S$ is continuous and $\Vert S\Vert\leq 1$. Observe that $1/p+1/q\geq 1$ is equivalent to $q\leq p^*$, 
where $p^*$ is the H\"older conjugate of~$p$ (i.e., $1/p+1/p^*=1$). Now, given $x=(x_i)_{i\in \Gamma}\in \ell_p(\Gamma)$ and 
$y=(y_i)_{i\in \Gamma}\in \ell_q(\Gamma)$, we have
$$
	\vert S(x,y)\vert\leq 
	\sum_{i\in F} \vert x_i\vert \vert y_i\vert\leq \left(\sum_{i\in F} \vert x_i\vert^p \right)^\frac{1}{p}\left(\sum_{i\in F}\vert y_i\vert^{p^*} 	\right)^\frac{1}{p^*}
$$
by H\"older's inequality. Since $q\leq p^*$ we get 
\[
	\left(\sum_{i\in F} \vert x_i\vert^p \right)^\frac{1}{p}\left(\sum_{i\in F}\vert y_i\vert^{p^*} \right)^\frac{1}{p^*} 
	\leq \left(\sum_{i\in F} \vert x_i\vert^p \right)^\frac{1}{p}\left(\sum_{i\in F}\vert y_i\vert^{q} \right)^\frac{1}{q}\\
	\leq \Vert x\Vert \Vert y\Vert.
\]
Consequently, $S$ is continuous and $\Vert S\Vert\leq 1$. Hence
\[
	\left\Vert \sum_{i\in F}\lambda_i e_i\otimes \tilde{e_i}\right\Vert\geq S\left(\sum_{i\in F}\lambda_i e_i\otimes \tilde{e_i}\right)=
	\sum_{i\in F} \lambda_i S(e_i,\tilde{e_i})=\sum_{i\in F} |\lambda_i|,
\]
which proves inequality ``$\geq$'' in~\eqref{eqn:lplq} and finishes the proof.
\end{proof}

The following definition fits in the general scheme of approximation properties of Banach spaces.

\begin{defi}\label{defi:BSAP}
We say that a Banach space~$Z$ has the {\em $\lambda$-bounded separable approximation property ($\lambda$-BSAP)}
for some $\lambda \geq 1$ if the identity operator on~$Z$ belongs to the SOT-closure of
$\lambda B_{\mathcal{L}(Z)} \cap \mathcal{S}(Z)$.
\end{defi}

The previous concept is a common extension of two properties which have been
thoroughly studied in the literature, namely, the {\em $\lambda$-bounded approximation property ($\lambda$-BAP)} and
the {\em $\lambda$-separable complementation property ($\lambda$-SCP)}. On the one hand, the $\lambda$-BAP is defined as
in Definition~\ref{defi:BSAP} by replacing $\mathcal{S}(Z)$ with the set of all finite rank operators on~$Z$. On the other hand,
a Banach space $Z$ is said to have the $\lambda$-SCP for some $\lambda \geq 1$ if for every
separable subspace $X \sub Z$ there is a projection $P \in \mathcal{S}(Z)$ with $\|P\|\leq \lambda$
such that $X \sub P(Z)$. Examples of Banach spaces with the $1$-SCP are WLD spaces and duals of Asplund spaces 
(see, e.g., \cite[Theorem~3.42]{fab-alt-JJ}).

\begin{lem}\label{lem:SCPidentity}
Let $Z$ be a Banach space. If $Z$ has the $\lambda$-SCP for some $\lambda\geq 1$,
then $Z$ has the $\lambda$-BSAP.
\end{lem}
\begin{proof} Let $I_Z:Z \to Z$ be the identity operator. 
Let $\{P_d: d\in D\}$ be the family of all projections belonging to $\lambda B_{\mathcal{L}(Z)} \cap \mathcal{S}(Z)$ and set $Z_d:=P_d(Z)$ for every $d\in D$.
Consider the preorder $\preceq$ in $D$ defined by 
$$
	d \preceq d' \ \Longleftrightarrow \ Z_d \sub Z_{d'}.
$$
Since $Z$ has the $\lambda$-SCP, $(D,\preceq)$ is a directed set. Clearly, the net $(P_d)_{d \in D}$ is 
SOT-convergent to~$I_Z$, as desired.
\end{proof}

\begin{defi}\label{defi:WS}
Let $X$ and $Y$ be Banach spaces. We denote by $\mathcal{W^*S}(X,Y^*)$ the set of 
all $T\in \mathcal{L}(X,Y^*)$ such that $T(B_{X})$ is $w^*$-separable.
\end{defi}

\begin{theo}\label{theo:LEMMA}
Let $X$ and $Y$ be Banach spaces and $V \sub \mathcal{L}(X,Y^*)$ be a convex set such that:
\begin{enumerate}
\item[(i)] $T\circ S \in V$ for every $S\in B_{\mathcal{L}(X)}$ and for every $T\in V$.
\item[(ii)] $V$ has convex countable tightness with respect to $W^*OT$. 
\end{enumerate}
Suppose that $X$ has the $\lambda$-BSAP for some $\lambda\geq 1$. The following statements hold: 
\begin{enumerate}
\item[(a)] Every element of~$V$ has $w^*$-separable range. 
\item[(b)] If $V \sub \mathcal{W}(X,Y^*)$, then $V \sub \mathcal{S}(X,Y^*)$.
\item[(c)] If $\lambda=1$, then $V \sub \mathcal{W^*S}(X,Y^*)$.
\end{enumerate}
\end{theo}
\begin{proof} Fix $T\in V$ and let us prove that $T(X)$ is $w^*$-separable.
Consider the map
$$
	\hat{T}: B_{\mathcal{L}(X)} \to V,
	\quad
	\hat{T}(S):=T \circ S.
$$
Let $I_X\in \mathcal{L}(X)$ be the identity operator and write 
$$
	W:=\hat{T}(\lambda B_{\mathcal{L}(X)}\cap \mathcal{S}(X)).
$$
Since $X$ has the $\lambda$-BSAP, the SOT-to-SOT continuity of~$\hat{T}$ implies that
$$
	T=\hat{T}(I_X) \in \overline{W}^{\text{SOT}} \sub \overline{W}^{\text{W$^*$OT}}.
$$
Therefore, since $W$ is convex and $V$ has convex countable tightness for W$^*$OT, we have 
$$
	T \in \overline{U}^{\text{W$^*$OT}}
$$
for some countable set $U \sub W$. 

Let $\mathcal{A} \sub \lambda B_{\mathcal{L}(X)}\cap \mathcal{S}(X)$
be a countable set such that $U=\hat{T}(\mathcal{A})$. For each $S\in \mathcal{A}$ we fix a countable set
$D_S \sub B_X$ such that $S(B_X) \sub \overline{S(D_S)}^{\|\cdot\|}$. Then
$$
	H:=\bigcup_{S\in \mathcal{A}}S(D_S)
$$
is a countable subset of~$\lambda B_X$ and so $T(H)$ is a countable subset of~$\lambda T(B_X)$. We claim that
\begin{equation}\label{eqn:ballH}
	T(B_X) \sub \overline{T(H)}^{w^*}.
\end{equation}

Indeed, take any $x\in B_X$. For each $S\in \mathcal{A}$ we have $S(x) \in \overline{S(D_S)}^{\|\cdot\|} \sub \overline{H}^{\|\cdot\|}$
and so $T(S(x)) \in \overline{T(H)}^{\|\cdot\|} \sub \overline{T(H)}^{w^*}$. Since $T\in \overline{\hat{T}(\mathcal{A})}^{\text{W$^*$OT}}$, we conclude
that
$$
	T(x) \in \overline{\{T(S(x)): \, S \in \mathcal{A}\}}^{w^*}\sub \overline{T(H)}^{w^*}.
$$
This proves~\eqref{eqn:ballH}. Hence $T(X) \sub \overline{T(\bigcup_{n\in \N} nH)}^{w^*}$
and so $T(X)$ is $w^*$-separable. 

(b) Since $T$ is weakly compact and $H \sub \lambda B_X$, the set $K:=\overline{T(H)}^{w}$ is weakly compact.
Observe that $K$ is norm separable because $H$ is countable. Moreover, since $K$ is weakly compact, it is also   
$w^*$-compact and $K=\overline{T(H)}^{w^*}$. From~\eqref{eqn:ballH} it follows that $T(B_X)$ is norm separable, that is, $T$ has 
norm separable range.

(c) Observe that if $\lambda=1$ then $H \sub B_X$ and \eqref{eqn:ballH} implies that $T(B_X)$ is $w^*$-separable.
\end{proof}

\begin{cor}\label{cor:SCPcct}
Let $X$ and $Y$ be Banach spaces such that $X\tenp Y$ has property~(C)
and $X$ has the $\lambda$-BSAP for some $\lambda\geq 1$. The following statements hold: 
\begin{enumerate}
\item[(a)] Every element of $\mathcal{L}(X,Y^*)$ has $w^*$-separable range.
\item[(b)] $\mathcal{W}(X,Y^*) \sub \mathcal{S}(X,Y^*)$.
\item[(c)] If $\lambda=1$, then $\mathcal{L}(X,Y^*) = \mathcal{W^*S}(X,Y^*)$.
\end{enumerate}  
\end{cor}
\begin{proof}
Apply Theorem \ref{theo:LEMMA} with $V=B_{\mathcal{L}(X,Y^*)}$ for statements (a) and (c) and with $V=B_{\mathcal{L}(X,Y^*)}\cap \mathcal{W}(X,Y^*)$ for statement (b).
\end{proof}

\begin{rem}\label{rem:W}
{\rm Part~(b) of Theorem~\ref{theo:LEMMA} and Corollary~\ref{cor:SCPcct} should be compared with the fact that}
for arbitrary Banach spaces $X$ and $Y$ we have 
$$
	\mathcal{W^*S}(X,Y^*) \cap \mathcal{W}(X,Y^*) \sub \mathcal{S}(X,Y^*).
$$ 
\end{rem}
\begin{proof}
Let $T \in \mathcal{W^*S}(X,Y^*) \cap \mathcal{W}(X,Y^*)$.
Since $T(B_X)$ is $w^*$-separable, the set $K = \overline{T(B_X)}^w$ is $w^*$-separable as well. 
Since $K$ is weakly compact, the weak and weak$^*$ topologies coincide on~$K$,
hence $K$ is weakly separable, which is equivalent to being norm separable. Thus, $T$ has separable range.
\end{proof}

\begin{cor}\label{cor:J}
Let $X$ be a non-separable Banach space having the $\lambda$-BSAP for some $\lambda\geq 1$.
Then $X\hat\otimes_\pi X^*$ fails property~(C).
\end{cor}
\begin{proof}
The canonical embedding $J: X\to X^{**}$ fails to have $w^*$-separable range, because the topologies $w^*$ and $w$ coincide on~$J(X)$
and $X$ is non-separable. The conclusion follows from Corollary~\ref{cor:SCPcct}(a).
\end{proof}

By Proposition \ref{pro:embedding-lp-lq}, if $1<p,q<\infty$ are such that $1/p+1/q \geq 1$ and $\Gamma$ is any non-empty set, 
then $\ell_p(\Gamma)\tenp \ell_q(\Gamma)$ contains a subspace isometric to~$\ell_1(\Gamma)$, and so it cannot have property~(C)
whenever $\Gamma$ is uncountable. The last statement is a particular case of the following consequence of Corollary~\ref{cor:SCPcct}(b):

\begin{cor}\label{cor:reflexive}
Let $X$ and $Y$ be Banach spaces such that $X$ is reflexive and $\mathcal{L}(X,Y^*)\neq \mathcal{S}(X,Y^*)$.
Then $X\tenp Y$ fails property~(C).
\end{cor}

As an application, we obtain the following classical result (see \cite[Theorem~2]{kal74}):

\begin{cor}[Kalton]\label{cor:Kalton}
Let $X$ be a Banach space.  If $\mathcal{L}(X)$ is reflexive, then $X$ is separable.
\end{cor}
\begin{proof}
Since $Y:=X^*$ is isometric to a subspace of~$\mathcal{L}(X)$, both $X$ and~$Y$ are reflexive.
We also have $(X\tenp Y)^*=\mathcal{L}(X)$, so that $X\tenp Y$ is reflexive as well.
The conclusion now follows from Corollary~\ref{cor:reflexive}.
\end{proof}

\section{WLD spaces and projective tensor products}\label{section:WLDpten}

Bearing in mind that the product of countably many Corson compact spaces is Corson, an argument similar to that of
Proposition~\ref{pro:separable-pten-C} yields the following:

\begin{pro}\label{pro:separable-pten-WLD}
Let $X$ and $Y$ be Banach spaces. If $X$ is separable and $Y$ is WLD, then $X\pten Y$ is WLD.
\end{pro}

This can also be obtained from the main result of this section, which is the following characterization of WLD projective tensor products:

\begin{theorem}\label{theorem:caractWLD}
Let $X$ and $Y$ be Banach spaces. The following statements are equivalent:
\begin{enumerate}
\item[(i)] $X\pten Y$ is WLD.
\item[(ii)] $X$ and $Y$ are WLD, $\mathcal{L}(X,Y^*)=\mathcal{S}(X,Y^*)$ and $\mathcal{L}(Y,X^*)=\mathcal{S}(Y,X^*)$.
\end{enumerate}
\end{theorem}

Our proof of Theorem~\ref{theorem:caractWLD} will use two lemmata:

\begin{lem}\label{cor:WLD-WS}
Let $X$ and $Y$ be WLD Banach spaces and $T\in \mathcal{L}(X,Y^*)$. Then $T(X)$ is separable if (and only if) it is $w^*$-separable.
\end{lem}
\begin{proof}
Define $Z:=\overline{T(X)}^{\|\cdot\|}$. If we consider $T$ as an operator from~$X$ to~$Z$, then $T$ has dense range
and so $T^*:Z^*\to X^*$ is injective. Therefore, $Z$ is WLD. 

Let $\Phi: Y^*\to \ell_\infty^c(\Gamma)$ be an
injective and $w^*$-to-$\tau_p(\Gamma)$ continuous operator, for some non-empty set~$\Gamma$. 
Since $Z$ is $w^*$-separable (because $T(X)$ is $w^*$-separable), the restriction $\Phi|_Z$ is an injective operator with values in $\ell_\infty(\Gamma_0)$ 
(as a subspace of $\ell_\infty^c(\Gamma)$) for some countable subset $\Gamma_0 \sub \Gamma$,
hence $Z^*$ is $w^*$-separable. It follows that $Z$ is separable (see, e.g., \cite[Proposition~5.40]{fab-alt-JJ}).
\end{proof}

\begin{lem}\label{lem:for2tensors}
Let $X$ and $Y$ be Banach spaces and let $\Gamma_X \sub B_X$ and $\Gamma_Y \sub B_Y$ be sets
such that $X=\overline{{\rm span}}(\Gamma_X)$ and $Y=\overline{{\rm span}}(\Gamma_Y)$.
Define $\Gamma:=\Gamma_X \times \Gamma_Y$ and
$$
	\Phi: \mathcal{B}(X,Y) = (X\pten Y)^* \to \ell_\infty(\Gamma)
$$	
by the formula
$$
	\Phi(S):=(S(x,y))_{(x,y)\in \Gamma}
	\quad
	\mbox{for all }S\in \mathcal{B}(X,Y).
$$
Then $\Phi$ is an operator which is injective and $w^*$-to-$\tau_p(\Gamma)$ continuous.
\end{lem}
\begin{proof}
Straightforward.
\end{proof}

\begin{proof}[Proof of Theorem~\ref{theorem:caractWLD}]
(i)$\impli$(ii) This follows from Corollary~\ref{cor:SCPcct} and Lemma~\ref{cor:WLD-WS}, because
the property of being WLD is hereditary, WLD spaces have property~(C) (see Subsection~\ref{subsection:WLD})
and the $1$-BSAP (combine \cite[Theorem~3.42]{fab-alt-JJ} and Lemma~\ref{lem:SCPidentity}).

(ii)$\impli$(i) We have to prove the existence of an injective and $w^*$-to-$\tau_p(\Gamma)$ continuous operator 
$\Phi:(X\pten Y)^*\to \ell_\infty^c(\Gamma)$ for certain non-empty set~$\Gamma$. 
Since $X$ is WLD, it admits a Markushevich basis $\{(x_i,x_i^*):i\in I\}$ and, for every $x^*\in X^*$, 
the set $\{i\in I:x^*(x_i)\neq 0\}$ is countable (see Subsection~\ref{subsection:WLD}).
By the same reason, $Y$ admits a Markushevich basis $\{(y_j,y_j^*):j\in J\}$ and, for every $y^*\in Y^*$, 
the set $\{j\in J:y^*(y_j)\neq 0\}$ is countable. We can assume that $\|x_i\| \leq 1$ for every $i\in I$ and that $\|y_j\|\leq 1$ for every $j\in J$. 

Write $\Gamma_X:=\{x_i:i\in I\}$ and $\Gamma_Y:=\{y_j:j\in J\}$. Define $\Gamma:=\Gamma_X\times \Gamma_Y$ and 
$$
	\Phi: \mathcal{B}(X,Y)=(X\pten Y)^* \to \ell_\infty(\Gamma)
$$ 
as in Lemma~\ref{lem:for2tensors}. Let us prove that $\Phi(S)\in \ell_\infty^c(\Gamma)$ for every $S\in (X\pten Y)^*$.
On the one hand, since $S_X$ has separable range, we have $S_X(X)\subseteq \overline{C}^{w^*}$ for some countable set $C =\{y_n^*:n\in \N\}\subseteq Y^*$. 
For each $n\in\mathbb N$ we choose a countable set $B_n\subseteq \Gamma_Y$ such that $y_n^*(y)=0$ whenever $y\in \Gamma_Y\setminus B_n$. 
Then $B:=\bigcup\limits_{n\in\mathbb N} B_n \sub \Gamma_Y$ is countable and $y_n^*(y)=0$ for every $y\in \Gamma_Y\setminus B$ and 
for every $n \in \N$. Since $S_X(X)\subseteq \overline{C}^{w^*}$ we deduce that
$$	
	S(x,y)=S_X(x)(y)=0\ \mbox{ for every }  x\in X \mbox{ and for every }y\in \Gamma_Y\setminus B.
$$
On the other hand, the same argument applied to~$S_Y$ ensures the existence of a countable set $A\subseteq \Gamma_X$ satisfying
$$
	S(x,y)=S_Y(y)(x)=0\ \mbox{ for every } x\in \Gamma_X\setminus A \mbox{ and for every } y\in Y.
$$
Clearly, $A\times B \sub \Gamma$ is countable and $S(x,y)=0$ whenever $(x,y)\in \Gamma \setminus A\times B$.
This shows that $\Phi(S)\in \ell_\infty^c(\Gamma)$. The proof is finished.
\end{proof}

We finish this section with an application of Theorem~\ref{theorem:caractWLD} to Lebesgue-Bochner spaces. 
Recall that if $\mu$ is a finite measure and $Y$ is a Banach space, then the Lebesgue-Bochner space $L_1(\mu,Y)$
is isometrically isomorphic to $L_1(\mu)\pten Y$ (see, e.g., \cite[Section~2.3]{rya}).
We will use the following result (see \cite[Corollary~2.4]{rod7}), which will also be needed in Section~\ref{section:WLDiten}.

\begin{theo}\label{theo:JR}
Let $Y$ be a WLD Banach space. Then $(B_{Y^*},w^*)$ has property (M) if and only if $\mathcal{L}(L_1(\mu),Y^*)=\mathcal{S}(L_1(\mu),Y^*)$
for every finite measure~$\mu$.
\end{theo}

\begin{cor}\label{cor:Lebesgue-Bochner}
Let $Y$ be a Banach space. The following statements are equivalent:
\begin{enumerate}
\item[(i)] $Y$ is WLD and $(B_{Y^*},w^*)$ has property (M).
\item[(ii)] $L_1(\mu,Y)$ is WLD for every finite measure~$\mu$.
\end{enumerate}
\end{cor}
\begin{proof} (ii)$\impli$(i) This follows from Theorems~\ref{theorem:caractWLD} and~\ref{theo:JR}.

(i)$\impli$(ii) We will apply Theorem~\ref{theorem:caractWLD}. On the one hand, any operator from $L_1(\mu)$ to~$Y^*$ has separable range by 
Theorem~\ref{theo:JR}. On the other hand, we claim that any $T\in \mathcal{L}(Y,L_\infty(\mu))$ has separable range. Indeed, if
we consider such a~$T$ as an operator from~$Y$ to $Z:=\overline{T(Y)}^{\|\cdot\|}$, then $T$ has dense range and so 
$T^*:Z^* \to Y^*$ is injective, hence $(B_{Z^*},w^*)$ is a Corson compact space with property~(M).

The compact space $L:=(B_{L_\infty(\mu)^*},w^*)$ admits a strictly positive measure, i.e., 
there is a regular Borel probability measure on~$L$ whose support is~$L$ (cf. \cite[Theorem~4.1]{ple2}).
Therefore, its continuous image $(B_{Z^*},w^*)$ admits a strictly positive measure as well.
Since $(B_{Z^*},w^*)$ is Corson and has property~(M), we conclude that it is metrizable, 
that is, $Z$ is separable. It follows that $T$ has separable range.
\end{proof}

\section{WCG spaces, their subspaces and projective tensor products}\label{section:WCGpten}

We begin this section with a result analogous to Propositions~\ref{pro:separable-pten-C} and~\ref{pro:separable-pten-WLD}
for WCG spaces:

\begin{pro}\label{pro:separable-pten-WCG}
Let $X$ and $Y$ be Banach spaces. If $X$ is separable and $Y$ is WCG, then $X\pten Y$ is WCG.
\end{pro}
\begin{proof}
Let $(x_n)_{n\in \N}$ be a dense sequence in~$X$ and let $K \sub Y$ be a weakly compact set such that
$Y=\overline{{\rm span}}(K)$. Then each $K_n:=\{x_n\}\otimes K$ is weakly compact in~$X\pten Y$ and
we have $X\pten Y=\overline{\bigcup_{n\in \N}K_n}^{\|\cdot\|}$, hence $X\pten Y$ is WCG.
\end{proof}

We next consider a slight extension of Theorem~\ref{theo:Diestel}, see Proposition~\ref{TheoAlmostDiestel} below. 
To this end we need to introduce more terminology.
Recall that a Banach space $X$ is said to have the \emph{Dunford-Pettis} property if $x_n^*(x_n)\to 0$
for all weakly null sequences $(x_n^*)_{n\in \N}$ in $X^*$ and $(x_n)_{n\in \N}$ in~$X$. This is equivalent to the fact that
every weakly compact operator from~$X$ to another Banach space is Dunford-Pettis (see, e.g., \cite[Proposition~13.42]{fab-ultimo}). 
As usual, an operator $T$ from~$X$ to a Banach space~$Z$ is said to be {\em Dunford-Pettis} (or {\em completely continuous}, shortly $T\in \mathcal{DP}(X,Z)$)
if $T(W)$ is norm compact whenever $W\sub X$ is weakly compact or, equivalently, if $(T(x_n))_{n\in \N}$
is norm convergent for every weakly Cauchy sequence $(x_n)_{n\in\N}$ in~$X$. 

\begin{defi}\label{defi:Ylimited}
Let $Y$ be a Banach space. A set $H \sub Y^*$ is said to be {\em $Y$-limited} if every 
weakly null sequence in~$Y$ converges uniformly on~$H$. 
\end{defi}

Given a Banach space~$Y$, it is immediate that a set $H \sub Y^*$ is $Y$-limited if and only if every countable subset of $H$ is $Y$-limited.
Furthermore, every relatively norm compact subset of~$Y^*$ is $Y$-limited. 
In fact, the property of being $Y$-limited is equivalent to being relatively compact with respect to the Mackey topology $\tau(Y^*,Y)$ 
(see, e.g., \cite[Theorem~3.11]{fab-alt-JJ}). If $Y$ contains no subspace isomorphic to~$\ell_1$, then 
every $Y$-limited subset of~$Y^*$ is relatively norm compact (see, e.g., \cite[Theorem~3.16]{fab-alt-JJ}).

\begin{defi}\label{defi:WDP}
Let $X$ and $Y$ be Banach spaces. We denote by $\WDP(X,Y^*)$ the set of all $T\in \mathcal{L}(X,Y^*)$ 
such that $T(W)$ is $Y$-limited for every weakly compact set $W \sub X$.
\end{defi}

\begin{rem}\label{rem:WDP}
Let $X$ and $Y$ be Banach spaces. The following statements hold:
\begin{enumerate}
\item[(i)] Let $T\in \mathcal{L}(X,Y^*)$. Then $T\in \WDP(X,Y^*)$ if and only if $\{T(x_n):n\in \N\}$ is $Y$-limited
for every weakly null sequence $(x_n)_{n\in \N}$ in~$X$.
\item[(ii)] $\mathcal{DP}(X,Y^*) \sub \WDP(X,Y^*)$. 
\item[(iii)] If $X$ or $Y$ has the Dunford-Pettis property, then  $\mathcal{L}(X,Y^*) = \WDP(X,Y^*)$.
\end{enumerate}
\end{rem}
\begin{proof} (i) is a consequence of the Eberlein-Smulyan theorem, whereas (ii) follows from the $Y$-limitedness
of relatively norm compact subsets of~$Y^*$.
For the proof of~(iii), fix $T\in \mathcal{L}(X,Y^*)$ and take any weakly null sequence $(x_n)_{n\in \N}$ in~$X$. If 
$\{T(x_n):n\in \N\}$ is not $Y$-limited, then there exist 
a weakly null sequence $(y_k)_{k\in \N}$, a map $\varphi:\N\to \N$ 
and $\epsilon>0$ in such a way that 
\begin{equation}\label{eqn:DP}
	\big|\langle T(x_{\varphi(k)}),y_k\rangle\big|=\big|\langle T^{*}(y_k),x_{\varphi(k)}\rangle\big| \geq \epsilon 
	\quad\text{for all $k\in \N$}.
\end{equation}
Since $(y_k)_{k\in \N}$ is weakly null, the set $\varphi(\N)$ is infinite and so, by passing to a subsequence, we can assume
that $\varphi$ is strictly increasing, hence $(x_{\varphi(k)})_{k\in\N}$ is weakly null.
Observe that the sequences $(T(x_{\varphi(k)}))_{k\in \N}$ and $(T^{*}(y_k))_{k\in\N}$ are weakly null in~$Y^*$ and~$X^*$, respectively.
Now, it is clear that \eqref{eqn:DP} contradicts that either $X$ or~$Y$ has the Dunford-Pettis property.
\end{proof}

\begin{pro}\label{TheoAlmostDiestel}
Let $X$ and $Y$ be Banach spaces with $\mathcal{L}(X,Y^*)=\WDP(X,Y^*)$. Then:
\begin{enumerate}
\item[(i)] $W_X\otimes W_Y$ is relatively weakly compact in $X\pten Y$ whenever $W_X \sub X$ and $W_Y \sub Y$ are 
relatively weakly compact.
\item[(ii)] $X\pten Y$ is WCG whenever $X$ and $Y$ are WCG.
\end{enumerate} 
\end{pro}
\begin{proof} (ii) is immediate from~(i). To prove~(i), by the Eberlein-Smulyan theorem,
it is enough to show that if $(x_n)_{n\in \N}$ and $(y_n)_{n\in\N}$ are weakly null sequences in $X$ and~$Y$, respectively,
then $(x_n \otimes y_n)_{n\in \N}$ is weakly null in $X\pten Y$. Take any $S \in (X \pten Y)^*$. Then $S_X\in \WDP(X,Y^*)$ and so 
$\{S_X(x_n):n\in \N\}$ is $Y$-limited. Therefore, $\langle S,x_n \otimes y_n\rangle=S_X(x_n)(y_n) \to 0$, as desired.
\end{proof}

\begin{rem}
\rm The equality $\mathcal{L}(X,Y^*)=\WDP(X,Y^*)$ is not necessary for $X\pten Y$ to be WCG. Indeed, the space $\ell_2 \pten \ell_2$ is separable (hence WCG), 
while $B_{\ell_2}$ is not $\ell_2$-limited and so the identity operator on~$\ell_2$ does not belong to $\WDP(\ell_2,\ell_2)$.
\end{rem}

Throughout the rest of this section we analyze the property of being 
{\em subspace} of a WCG space for projective tensor products. 

As we mentioned in Subsection~\ref{subsection:WCG}, a Banach space $X$
is subspace of a WCG space if and only if $(B_{X^*},w^*)$ is Eberlein compact.
Since the product of countably many Eberlein compact spaces is Eberlein (see, e.g., \cite[Theorem 1.2]{wag}), 
the same idea of Proposition~\ref{pro:separable-pten-C} gives the following:

\begin{pro}\label{pro:separable-pten-sWCG}
Let $X$ and $Y$ be Banach spaces. If $X$ is separable and $Y$ is subspace of a WCG space, then $X\pten Y$ is 
subspace of a WCG space.
\end{pro}

A key tool is the following characterization of subspaces of WCG spaces
due to Fabian, Montesinos and Zizler~\cite{fab-mon-ziz-2} (see, e.g., \cite[Theorem~6.13]{fab-alt-JJ}).
This result was the starting point of a fruitful branch of Banach space theory devoted to quantifications of already known results
(the concept of $\epsilon$-relatively weakly compact set is recalled in Definition~\ref{defi:wk} below).

\begin{theo}\label{theo:sWCG}
Let $Z$ be a Banach space. The following statements are equivalent:
\begin{enumerate}
\item[(i)] $Z$ is subspace of a WCG space.
\item[(ii)] For each $\epsilon>0$ there is a countable decomposition $B_{Z}=\bigcup_{n\in \N}M^\epsilon_n$ 
such that each $M_n^\epsilon$ is $\epsilon$-relatively weakly compact in~$Z$.
\end{enumerate}
\end{theo}

\begin{defi}\label{defi:wk}
Let $Z$ be a Banach space and $M \sub Z$ be a bounded set. 
\begin{enumerate}
\item[(i)] We write
$$
	{\rm wk}_Z(M):=\sup\big\{d(z^{**},Z): \, z^{**}\in \overline{M}^{w^*}\big\},
$$
where $d(z^{**},Z):=\inf\{\|z^{**}-z\|:z\in Z\}$ and $\overline{M}^{w^*}$ is the $w^*$-closure of~$M$ in~$Z^{**}$.
\item[(ii)] $M$ is said to be {\em $\epsilon$-relatively weakly compact}, for some $\epsilon\geq 0$, if 
$$
	\overline{M}^{w^*} \sub Z+\epsilon B_{Z^{**}}.
$$
\end{enumerate}
\end{defi}

\begin{rem}\label{rem:rwc}
\rm Under the assumptions of Definition~\ref{defi:wk}, we have:
$$
	 {\rm wk}_Z(M)< \epsilon 
	 \ \Longrightarrow \
	 \text{$M$ is $\epsilon$-relatively weakly compact} 
	 \ \Longrightarrow \
	 {\rm wk}_Z(M)\leq \epsilon. 
$$
\end{rem}

The following quantitative versions of the Dunford-Pettis property were introduced in~\cite{kac-alt}:

\begin{defi}\label{defi:KKS}
A Banach space $X$ is said to have the: 
\begin{enumerate}
\item[(i)] {\em Direct quantitative Dunford-Pettis property} if there is a constant $C>0$ such that
for every weakly null sequence $(x_n)_{n\in \N}$ in~$B_X$ and for every bounded sequence $(x_n^*)_{n\in \N}$ in~$X^*$ we have
$$
	\limsup_{n\to \infty}|x_n^*(x_n)|\leq C {\rm wk}_{X^*}(\{x_n^*:n\in \N\}).
$$

\item[(ii)] {\em Dual quantitative Dunford-Pettis property} if there is a constant $C>0$ such that
for every bounded sequence $(x_n)_{n\in \N}$ in~$X$ and for every weakly null sequence $(x_n^*)_{n\in \N}$ in~$B_{X^*}$ we have
$$
	\limsup_{n\to \infty}|x_n^*(x_n)|\leq C {\rm wk}_{X}(\{x_n:n\in \N\}).
$$
\end{enumerate}
\end{defi}

All $\mathcal{L}_1$ spaces and all $\mathcal{L}_\infty$ spaces have both the direct and dual
quantitative Dunford-Pettis properties (see \cite[Theorem~5.9]{kac-alt}).

\begin{theo}\label{theo:pten-sets}
Let $X$ and $Y$ be Banach spaces such that $X$ has the dual quantitative Dunford-Pettis property (with constant $C>0$).
If $W_X \sub B_X$ is $\epsilon$-relatively weakly compact for some $\epsilon\geq 0$
and $W_Y \sub B_Y$ is relatively weakly compact, then $W_X \otimes W_Y$ is $\epsilon'$-relatively weakly compact in $X\pten Y$
for $\epsilon'=(4C+2)\epsilon$.
\end{theo}
\begin{proof}
It suffices to check that $W_X \otimes W_Y$ $\epsilon'$-interchanges limits with~$B_{(X\pten Y)^*}$ (see, e.g., \cite[Theorem~3.69]{fab-alt-JJ}). 
Let $(x_n)_{n\in \N}$, $(y_n)_{n\in \N}$ and $(S_m)_{m\in \N}$ be sequences in $W_X$, $W_Y$ and $B_{(X\pten Y)^*}$, respectively, 
for which the iterated limits
$$
	\alpha:=\lim_{m\to\infty} \lim_{n\to\infty} \langle S_m, x_n\otimes y_n \rangle
	\quad\mbox{and}\quad
	\beta:=\lim_{n\to\infty} \lim_{m\to\infty} \langle S_m, x_n\otimes y_n \rangle
$$
exist. We will prove that $|\alpha-\beta|\leq \epsilon'$. Fix a subsequence $(y_{n_k})_{k\in \N}$ that is weakly convergent to some $y\in Y$.

{\em Step~1}. Take any $S\in B_{(X\pten Y)^*}$ for which $\lim_{n\to\infty} \langle S,x_n\otimes y_n\rangle$ exists.
Then the fact that $X$ has the dual quantitative Dunford-Pettis property implies that
\begin{equation}\label{eqn:KKS}
	\limsup_{k\to\infty} \big|\langle S, x_{n_k}\otimes y_{n_k} \rangle - \langle S,x_{n_k}\otimes y \rangle\big|
	=\limsup_{k\to\infty} \big|\langle x_{n_k},S_{X}^*(y_{n_k}-y)\rangle\big| \leq 2C\epsilon.
\end{equation}
Let $x^{**}$ be a $w^*$-cluster point of~$(x_{n_k})_k$ in~$X^{**}$. Then
$\langle x^{**},S_X^*(y)\rangle$ is a cluster point of the sequence $(\langle S,x_{n_k}\otimes y\rangle)_{k\in \N}$ and so
\eqref{eqn:KKS} yields
\begin{equation}\label{eqn:KKS2}
	\left|\lim_{n\to\infty} \langle S,x_{n}\otimes y_{n}\rangle-\langle x^{**},S_X^*(y)\rangle\right|=
	\left|\lim_{k\to\infty} \langle S,x_{n_k}\otimes y_{n_k}\rangle-\langle x^{**},S_X^*(y)\rangle\right|\leq 2C\epsilon.
\end{equation}
Since $W_X$ is $\epsilon$-relatively weakly compact, there is $x\in X$ such that $\|x^{**}-x\|\leq \epsilon$, hence
\begin{multline}\label{eqn:Grothendieck}
	\left|\lim_{n\to\infty} \langle S,x_{n}\otimes y_{n}\rangle-\langle S,x\otimes y\rangle\right| \\ =
	\left|\lim_{n\to\infty} \langle S,x_{n}\otimes y_{n}\rangle-\langle x,S_X^*(y)\rangle \right|
	\stackrel{\eqref{eqn:KKS2}}{\leq} (2C+1)\epsilon.
\end{multline}

{\em Step~2.} Let $\tilde{S}$ be any $w^*$-cluster point of $(S_m)_{m\in \N}$ in $B_{(X\pten Y)^*}$. 
Inequality~\eqref{eqn:Grothendieck} applied to each $S_m$ yields
$$
	\left|\lim_{n\to\infty} \langle S_m,x_{n}\otimes y_{n}\rangle-\langle S_m,x\otimes y\rangle\right|\leq (2C+1)\epsilon
	\quad\mbox{for all }m\in \N
$$
and so
$$
	|\alpha-\langle \tilde{S},x\otimes y\rangle|\leq (2C+1)\epsilon.
$$
Observe that $	\beta=\lim_{n\to\infty} \langle \tilde{S},x_n\otimes y_n\rangle$, 
so another appeal to inequality~\eqref{eqn:Grothendieck} (now applied to~$\tilde{S}$) yields
$$
	|\beta-\langle\tilde{S},x\otimes y\rangle |\leq (2C+1)\epsilon.
$$
It follows that $|\alpha-\beta|\leq (4C+2)\epsilon$, as required.
\end{proof}

\begin{rem}\label{rem:pten-sets}
\rm The same statement holds if the assumption is replaced by ``$Y$ has the direct quantitative Dunford-Pettis property (with constant $C>0)$''.
Indeed, the proof follows the same steps and the only difference is that inequality~\eqref{eqn:KKS} is obtained using~$S_Y$ and
bearing in mind that $S_Y^*(W_X)$ is $\epsilon$-relatively weakly compact in~$Y^*$
(as it can be easily checked).
\end{rem}

\begin{theo}\label{theo:pten-subspaceWCG-qDP}
Let $X$ and $Y$ be Banach spaces such that either $X$ has the dual quantitative Dunford-Pettis property or $Y$ has the direct
quantitative Dunford-Pettis property.
If $X$ is subspace of a WCG space and $Y$ is WCG, then $X \pten Y$ is subspace of a WCG space. 
\end{theo}
\begin{proof}
Let $C>0$ be a constant witnessing that $X$ has the dual quantitative Dunford-Pettis property
or $Y$ has the direct quantitative Dunford-Pettis property.
Fix $\epsilon>0$. Since $X$ is subspace of a WCG space, 
for each $m\in \N$ there is a sequence $(B^m_n)_{n\in \N}$ of $\frac{\epsilon}{4(4C+2)m}$-relatively weakly compact
sets in such a way that $B_X=\bigcup_{n\in \N}B^m_n$ (apply Theorem~\ref{theo:sWCG}). 
Clearly, for any $n,m\in \N$ the set $A_{n,m}:=mB^m_n$ is $\frac{\epsilon}{4(4C+2)}$-relatively weakly compact
and $X=\bigcup_{n,m\in \N}A_{n,m}$. This shows that $X$ can be covered by countably many $\frac{\epsilon}{4(4C+2)}$-relatively weakly compact
sets. Fix a sequence $(B_n)_{n\in \N}$ of $\frac{\epsilon}{4(4C+2)}$-relatively weakly compact subsets of~$X$ such that $X=\bigcup_{n\in \N}B_n$.
We can assume without loss of generality that $B_n \sub B_{n+1}$ for all $n\in \N$.

Now, we use the fact that $Y$ is WCG to choose a weakly compact absolutely convex set
$K\sub Y$ such that $\bigcup_{n\in \N}nK$ is dense in~$Y$. By Theorem~\ref{theo:pten-sets} and Remark~\ref{rem:pten-sets},
for each $n\in \N$ the set $B_n\otimes nK$ is $\frac{\epsilon}{4}$-relatively weakly compact in $X\pten Y$,
hence ${\rm co}(B_n \otimes nK)$ is $\frac{\epsilon}{2}$-relatively weakly compact
(see, e.g., \cite[Theorem~3.64]{fab-alt-JJ}) and so
$$
	C_n:= {\rm co}(B_n \otimes nK) + \frac{\epsilon}{2}B_{X\pten Y}
$$ 
is $\epsilon$-relatively weakly compact (as it can be easily checked). 

Observe that $X\pten Y=\bigcup_{n\in \N}C_n$. Indeed, if $u\in X\pten Y$, then there exist
$x_1,\dots,x_p\in \|u\| B_X$, $y_1,\dots,y_p \in B_Y$ and $\lambda_1,\dots,\lambda_p\geq 0$ with $\sum_{i=1}^p \lambda_i=1$
such that $\|u-\sum_{i=1}^p \lambda_i x_i\otimes y_i\|\leq\frac{\epsilon}{4}$ (see, e.g., \cite[Proposition~2.2]{rya}).
Take $n\in \N$ large enough such that $x_i\in B_n$ and $y_i\in nK+\frac{\epsilon}{4\|u\|}B_Y$ for all $i\in \{1,\dots,p\}$,
and pick $\tilde{y}_i\in nK$ such that $\|y_i-\tilde{y}_i\|\leq \frac{\epsilon}{4\|u\|}$.
Then $\|u-\sum_{i=1}^p \lambda_i x_i\otimes \tilde{y}_i\|\leq\frac{\epsilon}{2}$
and so $u\in C_n$. Therefore, we have $B_{X\pten Y}=\bigcup_{n\in \N}C_n\cap B_X$, with
each $C_n\cap B_X$ being $\epsilon$-relatively weakly compact. Another appeal to Theorem~\ref{theo:sWCG}
ensures that $X\pten Y$ is subspace of a WCG space.
\end{proof}

Our next objective is to show that if $X$ and $Y$ are subspaces of WCG spaces and
$\mathcal{L}(X,Y^*)=\mathcal{K}(X,Y^*)$, then $X\pten Y$ is subspace of a WCG space
(Corollary~\ref{cor:pten-subspaceWCG-qDPalloperators}). We will obtain this as a consequence of a technical
result (Theorem~\ref{theo:pten-subspaceWCG-qDPoperator}) which might be of independent interest.
We first need to introduce some terminology. 

\begin{defi}
Let $X$ be a Banach space.
\begin{enumerate}
\item[(i)] Let $(x_n)_{n\in \N}$ be a bounded sequence in~$X$. We write
$$
	{\rm ca}((x_n)_{n\in \N}):=\inf_{m\in \N}\sup_{n,n'\geq m}\|x_n-x_{n'}\|
$$ 
and
$$
	\delta((x_{n})_{n\in \N}) = \sup_{x^*\in B_{X^*}}
	\inf_{m\in \N}\sup_{n,n'\geq m}|x^*(x_n)-x^*(x_{n'})|.
$$
\item[(ii)] We say that a set $M \sub X$ is {\em $\epsilon$-precompact} (resp., {\em $\epsilon$-weakly precompact}), 
for some $\epsilon\geq 0$, if it is bounded and 
every sequence $(x_n)_{n\in \N}$ in~$M$ admits a subsequence~$(x_{n_k})_{k\in \N}$
such that ${\rm ca}((x_{n_k})_{k\in \N}) \leq \epsilon$ (resp., $\delta((x_{n_k})_{k\in \N})\leq \epsilon$).
\end{enumerate}
\end{defi}

We will also need the following quantitative strengthening of the usual notion of Dunford-Pettis operator:

\begin{defi}\label{defi:operator-qDP}
Let $X$ and $Z$ be Banach spaces and let $c>0$. We say that an operator $T:X \to Z$ is 
{\em $c$-Dunford-Pettis} if, for each $\epsilon\geq 0$ and for each $\epsilon$-weakly precompact set $W \sub X$,
the set $T(W)$ is $c\epsilon$-precompact. We denote by $\mathcal{DP}_c(X,Z)$ the set of all $c$-Dunford-Pettis operators
from~$X$ to~$Z$. 
\end{defi}

\begin{exa}\label{exa:qDP}
\rm Examples of $c$-Dunford-Pettis operators are:
\begin{enumerate}
\item[(i)] Compact operators.
\item[(ii)] Absolutely summing operators. Indeed, combine Pietsch's factorization theorem 
(see, e.g., \cite[2.13]{die-alt}) and \cite[Lemma~2.8]{rod20}. 
\item[(iii)] Any operator from/to a Banach space with the so called quantitative Schur property (see \cite{kal-spu-2,kal-spu-3}).
\end{enumerate}
\end{exa}

In order to prove Theorem~\ref{theo:pten-subspaceWCG-qDPoperator} we will need 
the following characterization of subspaces of WCG spaces (see, e.g., \cite[Theorem~6.13]{fab-alt-JJ}).

\begin{theo}\label{theo:Eberlein}
Let $X$ be a Banach space and let $G \sub X$ be a set such that $X=\overline{{\rm span}}(G)$ and, 
for each $x^*\in X^*$, the set $\{x\in G:x^*(x)\neq 0\}$ is countable. The following statements are equivalent:
\begin{enumerate}
\item[(i)] $X$ is subspace of a WCG space.
\item[(ii)] For each $\varepsilon>0$ there exists a countable decomposition 
$G=\bigcup\limits_{n\in\mathbb N} G_n^\varepsilon$ such that 
$\big\{x \in G_n^\varepsilon: \, |x^*(x)| > \varepsilon\big\}$
is finite for every $n\in\mathbb N$ and for every $x^* \in B_{X^*}$.
\end{enumerate}
\end{theo}

\begin{rem}
\rm Observe that condition~(ii) in Theorem~\ref{theo:Eberlein} implies that, for each $x^*\in X^*$, the set $\{x\in G:x^*(x)\neq 0\}$ is countable.
\end{rem}

\begin{theo}\label{theo:pten-subspaceWCG-qDPoperator}
Let $X$ and $Y$ be Banach spaces which are subspaces of WCG spaces. 
Suppose that $\mathcal{L}(X,Y^*)=\mathcal{DP}_c(X,Y^*)$ and $\mathcal{L}(Y,X^*)=\mathcal{DP}_c(Y,X^*)$
for some $c>0$. Then $X\pten Y$ is subspace of a WCG space.
\end{theo}
\begin{proof}
Both $X$ and $Y$ are WLD and so they admit Markushevich bases, say 
$$
	\{(x_i,x_i^*):i\in \Gamma_1\} \sub X\times X^*
	\quad\mbox{and}\quad 
	\{(y_j,y_j^*):j\in \Gamma_2\} \sub Y\times Y^*.
$$
Moreover, for every $x^*\in X^*$ (resp. $y^*\in Y^*$), the set $\{i\in \Gamma_1:x^*(x_i)\neq 0\}$ 
(resp., $\{j\in \Gamma_2:y^*(y_j)\neq 0\}$) is countable (see Subsections~\ref{subsection:WCG} and~\ref{subsection:WLD}). 
We can assume that $\|x_i\| \leq 1$ and that $\|y_j\|\leq 1$ for every $(i,j)\in \Gamma_1\times \Gamma_2$.

Define $G:=\{x_i\otimes y_j:(i,j)\in \Gamma_1\times \Gamma_2\}\sub X\pten Y$. Clearly, $X\pten Y=\overline{{\rm span}}(G)$. We will check that $G$ satisfies 
condition~(ii) in Theorem~\ref{theo:Eberlein} and, therefore, $X\pten Y$ is subspace of a WCG space.

Fix $\varepsilon>0$. Pick any $\varepsilon>\varepsilon'>0$. By Theorems~\ref{theo:sWCG}
and~\ref{theo:Eberlein}, we can find countable decompositions 
$$
	B_X=\bigcup\limits_{n\in\mathbb N} B_{1,n},
	\quad
	B_Y=\bigcup\limits_{n\in\mathbb N} B_{2,n},
	\quad
	\Gamma_1=\bigcup\limits_{n\in\mathbb N} \Gamma_{1,n}
	\quad\mbox{and}\quad
	\Gamma_2=\bigcup\limits_{n\in\mathbb N} \Gamma_{2,n}
$$ 
such that, for each $n\in \Nat$, the sets $B_{1,n}$ and $B_{2,n}$ are $\frac{\epsilon'}{4c}$-relatively weakly compact
and the sets
\begin{eqnarray*}
	U(x^*,n) & := & \left\{i \in \Gamma_{1,n}: \, | x^*(x_i)|> \frac{\varepsilon}{2}\right\} \\
	V(y^*,n) & := & \left\{j \in \Gamma_{2,n}: \, | y^*(y_j)|> \frac{\varepsilon}{2}\right\}
\end{eqnarray*} 
are finite for every $x^* \in B_{X^*}$ and for every $y^*\in B_{Y^*}$. 
Define 
$$
	G_{(n_1,m_1,n_2,m_2)}^\varepsilon:=
	\{x_i\otimes y_j:\, (x_i,y_j)\in B_{1,n_1} \times B_{2,n_2} \mbox{ and } (i,j)\in \Gamma_{1,m_1} \times \Gamma_{2,m_2}\}
$$ 
for all $(n_1,m_1,n_2,m_2) \in\mathbb N^4$. Let us prove that the countable decomposition 
$$
	G=\bigcup\limits_{(n_1,m_1,n_2,m_2)\in\mathbb N^4} G_{(n_1,m_1,n_2,m_2)}^\epsilon
$$ 
satisfies the required property.

Fix $S\in B_{(X \pten Y)^*}$ and $(n_1,m_1,n_2,m_2)\in \N^4$. 
Since $B_{1,n_1}$ is $\frac{\epsilon'}{4c}$-relatively weakly compact, 
it is $\frac{\epsilon'}{2c}$-weakly precompact (see \cite[Lemma~3.7]{rod20})
and therefore $S_X(B_{1,n_1})$ is an $\frac{\epsilon'}{2}$-precompact subset of~$B_{Y^*}$ 
(because $S_X$ is a $c$-Dunford-Pettis operator with $\|S_X\|\leq 1$ and $B_{1,n_1}\sub B_X$). 
Therefore, since $\varepsilon>\varepsilon'$, we can 
find finitely many $y_1^*,\ldots, y_p^*\in B_{Y^*}$
such that 
$$
	S_X(B_{1,n_1})\subseteq \bigcup\limits_{k=1}^p B\left(y_k^*,\frac{\epsilon}{2}\right), 
$$
where $B(y_k^*,\frac{\epsilon}{2})$ denotes the closed ball of~$Y^*$ centered at~$y_k^*$ with radius~$\frac{\epsilon}{2}$.
Analogously, there exist finitely many $x_1^*,\ldots, x_q^*\in B_{X^*}$ such that 
$$
	S_Y(B_{2,n_2})\subseteq \bigcup\limits_{l=1}^q B\left(x_l^*,\frac{\epsilon}{2}\right).
$$
Note that 
$$
	H:=\left(\bigcup_{l=1}^q U(x_l^*,m_1) \right) \times \left(\bigcup_{k=1}^p V(y_k^*,m_2)\right) \sub \Gamma_1\times \Gamma_2
$$ 
is finite. In order to finish the proof we will show that 
$$
	\big|\langle S,x_i\otimes y_j\rangle\big| \leq \varepsilon
	\quad\mbox{for every } x_i\otimes y_j \in G_{(n_1,m_1,n_2,m_2)}^\epsilon \mbox{ with }(i,j)\not\in H.
$$ 
To this end, suppose for instance that $i\notin \bigcup_{l=1}^q U(x_l^*,m_1)$ (the other case runs similarly). 
Since $y_j \in B_{2,n_2}$, there is $l\in \{1,\dots,q\}$ such that 
$\Vert x_l^*-S_Y(y_j)\Vert \leq \frac{\varepsilon}{2}$. 
Since $i\not\in U(x_l^*,m_1)$, we have $\vert x_l^*(x_i)\vert\leq \frac{\varepsilon}{2}$ and therefore
$$
	\vert \langle S, x_i\otimes y_j\rangle|=\vert \langle S_Y(y_j),x_i \rangle \vert
	\leq \Vert x_l^*-S_Y(y_j)\Vert + \vert x_l^*(x_i)\vert\leq \varepsilon,
$$
as required. The proof is finished.
\end{proof}

It is well known that, for arbitrary Banach spaces~$X$ and~$Y$, 
the equalities $\mathcal{L}(X,Y^*)=\mathcal{K}(X,Y^*)$ and $\mathcal{L}(Y,X^*)=\mathcal{K}(Y,X^*)$
are equivalent. Indeed, this is a consequence of Schauder's theorem (saying an operator is compact if and only if its adjoint is compact)
and the fact that every $T\in \mathcal{L}(Y,X^*)$ coincides with the restriction of $(T^*|_X)^*$ to~$Y$. 
 
From Theorem~\ref{theo:pten-subspaceWCG-qDPoperator} we get:

\begin{cor}\label{cor:pten-subspaceWCG-qDPalloperators}
Let $X$ and $Y$ be Banach spaces which are subspaces of WCG spaces. If
$\mathcal{L}(X,Y^*)=\mathcal{K}(X,Y^*)$, then $X\pten Y$ is subspace of a WCG space.
\end{cor}

\section{Hilbert generated spaces, their subspaces and projective tensor products}\label{section:Hilbertpten}

In this section we study the stability under projective tensor products of the property
of being subspace of a Hilbert generated space (recalled in Subsection~\ref{subsection:WCG})
for some concrete Banach spaces. We begin with a general result.

\begin{pro}\label{pro:separable-pten-HG}
Let $X$ and $Y$ be Banach spaces. If $X$ is separable and $Y$ is Hilbert generated (resp., subspace of a Hilbert generated space), then
$X\pten Y$ is Hilbert generated (resp., subspace of a Hilbert generated space).
\end{pro}
\begin{proof} Suppose first that $Y$ is Hilbert generated and fix an operator $T:\ell_2(\Gamma) \to Y$ with dense range, 
for some non-empty set~$\Gamma$. Let $(x_n)_{n\in \N}$ be a dense sequence in~$B_X$ and consider the operator
$$
	\tilde{T}: \ell_2(\ell_2(\Gamma)) \to X\pten Y,
	\quad
	\tilde{T}\big((u_n)_{n\in \N}\big):= \sum_{n\in \N}\frac{1}{n} x_n \otimes T(u_n),
$$	
where  $\ell_2(\ell_2(\Gamma))$ stands for the $\ell_2$-sum of countably many copies of~$\ell_2(\Gamma)$, i.e., 
the Banach space of all sequences $(u_n)_{n\in\N}$ in~$\ell_2(\Gamma)$ such that $(\|u_n\|)_{n\in\N} \in \ell_2$. It is immediate that 
$\tilde{T}$ has dense range, so $X\pten Y$ is Hilbert generated.

The argument for subspaces of Hilbert generated spaces follows the same lines
of Proposition~\ref{pro:separable-pten-C}, bearing in mind that a Banach space~$Z$ is subspace of a Hilbert generated space
if and only if $(B_{Z^*},w^*)$ is uniform Eberlein compact (see Subsection~\ref{subsection:WCG}) and
the fact that uniform Eberlein compactness is preserved by countable products (see, e.g., \cite[Theorem~3.6]{wag}).
\end{proof}

Clearly, the space $\ell_p(\Gamma)$ is Hilbert generated for any $2\leq p < \infty$ and for any non-empty set~$\Gamma$, but it 
fails to be Hilbert generated when $1<p<2$ and $\Gamma$ is uncountable (see \cite[Lemma~6]{fab-alt-J-4}). 
Still in this case $\ell_p(\Gamma)$ is subspace of a Hilbert generated space, because so is every superreflexive space. 
If $1<p,q<\infty$ satisfy $1/p+1/q<1$, then the projective tensor product $\ell_p(\Gamma) \pten \ell_q(\Gamma)$
is reflexive (as we mentioned in the introduction) but cannot be superreflexive unless $\Gamma$ is finite (see \cite[p.~522]{bu-die}).
We next prove that it is always subspace of a Hilbert generated space:

\begin{theo}\label{theo:lp-pten-lq-sHG}
Let $\Gamma$ and $\Delta$ be non-empty sets and let $1< p,q<\infty$ be such that $1/p+1/q<1$. 
Then	$\ell_p(\Gamma)\pten \ell_q(\Delta)$ is subspace of a Hilbert generated space.
\end{theo}

The proof of Theorem~\ref{theo:lp-pten-lq-sHG} uses the following elementary lemma:

\begin{lem}\label{lem:cansino}
Let $I$ and $J$ be sets, $\Omega \sub I\times J$ and $m,r,s\in \N$. Suppose that:
\begin{enumerate}
\item[(i)] $|\{j\in J: (i,j)\in \Omega\}|\leq r$ for every $i\in I$;
\item[(ii)] $|\{i\in I: (i,j)\in \Omega\}|\leq s$ for every $j\in J$;
\item[(iii)] $|\Omega| \geq m(r+s)$.
\end{enumerate}
Then there exist sets $\{i_1,\dots,i_m\}\sub I$ and $\{j_1,\dots,j_m\}\sub J$ with cardinality~$m$ such that
$(i_k,j_k)\in \Omega$ for all $k\in \{1,\dots,m\}$.
\end{lem}
\begin{proof}
For each $(i,j)\in I\times J$, write
$$
	\Omega_{(i,j)}:=\{(i',j')\in \Omega: \, \text{$i=i'$ or $j=j'$}\}, 
$$
so that $|\Omega_{(i,j)}|\leq r+s$. Therefore, we have 
$$
	\Omega \setminus \bigcup_{(i,j)\in F}\Omega_{(i,j)} \neq \emptyset
$$
for every $F \sub I\times J$ with $|F|< m$. Now, we can apply this fact recursively to get $(i_1,j_1), \dots, (i_m,j_m) \in \Omega$
in such a way that $(i_{k'},j_{k'})\not \in \bigcup_{k<k'} \Omega_{(i_k,j_k)}$ for all $k'\leq m$, hence
$i_k\neq i_{k'}$ and $j_k\neq j_{k'}$ whenever $k\neq k'$.
\end{proof}

Another key ingredient is the following characterization of subspaces of Hilbert generated spaces
(see, e.g., \cite[Theorem~6.30]{fab-alt-JJ}) which should be compared with Theorem~\ref{theo:Eberlein}:

\begin{theo}\label{theo:sHG}
Let $X$ be a Banach space. The following statements are equivalent:
\begin{enumerate}
\item[(i)] $X$ is subspace of a Hilbert generated space.
\item[(ii)] There is a set $G\subseteq B_X$ with $X=\overline{{\rm span}}(G)$ such that for every $\varepsilon>0$ there is a countable decomposition 
$G=\bigcup_{n\in \N} G_n^\varepsilon$ such that 
$$
	\big|\{x\in G_n^\varepsilon : |x^*(x)|>\varepsilon\}\big|<n 
	\quad \text{for every $n\in\mathbb{N}$ and for every $x^\ast\in B_{X^\ast}$}.
$$
\end{enumerate}
\end{theo}

\begin{proof}[Proof of Theorem~\ref{theo:lp-pten-lq-sHG}]
We can assume without loss of generality that $\Gamma=\Delta$. Indeed, observe that if $|\Delta|\leq |\Gamma|$, then
$\ell_q(\Delta)$ is a $1$-complemented subspace of $\ell_q(\Gamma)$
and so $\ell_p(\Gamma)\pten \ell_q(\Delta)$ embeds isometrically into $\ell_p(\Gamma)\pten \ell_q(\Gamma)$
(see, e.g., \cite[Proposition~2.4]{rya}).

Let $\{e_\gamma:\gamma\in \Gamma\}$ and $\{\tilde{e_\gamma}:\gamma\in \Gamma\}$
be the canonical bases of $\ell_p(\Gamma)$ and $\ell_q(\Gamma)$, respectively.
The set $G:=\{e_\gamma\otimes \tilde{e_{\gamma'}}: (\gamma,\gamma')\in \Gamma\times \Gamma\} \sub B_{\ell_p(\Gamma)\pten \ell_q(\Gamma)}$ 
satisfies $\ell_p(\Gamma)\pten \ell_q(\Gamma)=\overline{{\rm span}}(G)$.
By Theorem~\ref{theo:sHG}, in order to show that $\ell_p(\Gamma)\pten \ell_q(\Gamma)$ is subspace of a Hilbert generated
space it is enough to prove that for every $\varepsilon>0$ there exists $n\in\mathbb{N}$ 
such that for all $S \in B_{(\ell_p(\Gamma)\pten \ell_q(\Gamma))^\ast}$ we have
$$
	\big|\{(\gamma,\gamma')\in \Gamma \times \Gamma : \, |S(e_\gamma \otimes \tilde{e_{\gamma'}})| > \varepsilon \}\big|<n.
$$
Suppose by contradiction that there exists $\varepsilon>0$ such that for each $n\in\mathbb{N}$ there is 
$S^n\in B_{(\ell_p(\Gamma)\pten \ell_q(\Gamma))^\ast}$ in such a way that the set 
$$
	\Omega_n:=\{(\gamma,\gamma')\in \Gamma \times \Gamma : \, |S^n(e_\gamma\otimes \tilde{e_{\gamma'}})| > \varepsilon \}
$$
has cardinality $|\Omega_n|\geq n$. Given any $n\in \N$, 
write $T_n:=S_{\ell_p(\Gamma)}^n \in B_{\mathcal{L}(\ell_p(\Gamma),\ell_{q^*}(\Gamma))}$ and notice that for each $\gamma\in \Gamma$ 
the set $\{\gamma'\in \Gamma: (\gamma,\gamma')\in \Omega_n\}$
has cardinality $\leq \varepsilon^{-q^*}$, because
$$
	1\geq \|T_n(e_\gamma)\|^{q^*}\geq \sum_{\gamma'\in \Gamma: \, (\gamma,\gamma')\in \Omega_n}|T_n(e_\gamma)(\tilde{e_{\gamma'}})|^{q^*}
	\geq \epsilon^{q^*} |\{\gamma'\in \Gamma: (\gamma,\gamma')\in \Omega_n\}|.
$$
Similarly, for each $\gamma'\in \Gamma$ 
the set $\{\gamma\in \Gamma: (\gamma,\gamma')\in \Omega_n\}$
has cardinality $\leq \varepsilon^{-p^*}$.

Choose $t\in \N$ large enough such that $t\geq \varepsilon^{-p^*}+ \varepsilon^{-q^*}$.
Fix $m\in \N$ and write $n:= mt$. By Lemma~\ref{lem:cansino},
we can find sets $\{i_1,\ldots,i_m\}\sub \Gamma$ and $\{j_1,\ldots,j_m\}\sub\Gamma$, both of cardinality~$m$, such that 
\begin{equation}\label{eqn:IJ}
	|T_n(e_{i_k})(\tilde{e_{j_k}})| > \varepsilon
	\quad \text{for all $k\in\{1,\dots,m\}$}. 
\end{equation}

Let $T:\mathbb{R}^m \to \mathbb{R}^m$ be the linear map whose matrix is 
$$
	\left(
	\begin{array}{cccc} 
	T_n(e_{i_1})(\tilde{e_{j_1}}) & T_n(e_{i_2})(\tilde{e_{j_1}}) & \dots & T_n(e_{i_m})(\tilde{e_{j_1}}) \\ 
	T_n(e_{i_1})(\tilde{e_{j_2}}) & T_n(e_{i_2})(\tilde{e_{j_2}}) & \dots & T_n(e_{i_m})(\tilde{e_{j_2}}) \\
	\vdots & \vdots & \ddots &\vdots \\
	T_n(e_{i_1})(\tilde{e_{j_m}}) & T_n(e_{i_2})(\tilde{e_{j_m}}) & \dots & T_n(e_{i_m})(\tilde{e_{j_m}})
	\end{array}
	\right). 
$$
Then $T$ can be seen as an operator from $\ell^m_\infty$ to $\ell^m_1$ and also as an operator from~$\ell^m_p$ to~$\ell^m_{q^*}$.
Notice that $\|T\|_{\mathcal{L}(\ell^m_p,\ell^m_{q^*})}\leq 1$, because it factors as
$$
	\xymatrix@R=3pc@C=3pc{\ell^m_p
	\ar[r]^{T} \ar[d]_{u} & \ell^m_{q^*}\\
	\ell_p(\Gamma)  \ar[r]^{T_n}  & \ell_{q^*}(\Gamma) \ar[u]_{\pi} \\
	}
$$
where $u$ and $\pi$ are (up to the natural isometric isomorphisms) 
the inclusion of $\overline{{\rm span }}(\{e_{i_1}, \ldots, e_{i_m}\})$ 
and the projection onto $\overline{{\rm span }}(\{\hat{e_{j_1}}, \ldots, \hat{e_{j_m}}\})$, respectively
(here $\{\hat{e_\gamma}:\gamma\in \Gamma\}$ denotes the canonical basis of~$\ell_{q^*}(\Gamma)$).
Since the identity operator $\ell^m_\infty \to \ell^m_p$ has norm $\leq m^{1/p}$
and the identity operator $\ell^m_{q^*} \to \ell_1^m$ has norm $\leq m^{1/q}$ (as an application of Hölder's inequality),
we conclude that
\begin{equation}\label{eqn:norm01}
	\|T\|_{\mathcal{L}(\ell^m_{\infty},\ell^m_{1})} \leq m^{1/p+1/q}.
\end{equation}
		
Grothendieck's inequality (see, e.g., \cite[1.14]{die-alt}) applied to the matrix above yields
$$
	\left|\sum_{k=1}^m \sum_{k'=1}^m T_n(e_{i_k})(e_{j_{k'}}) \langle u_{k},v_{k'}\rangle\right| \leq K_G \|T\|_{\mathcal{L}(\ell^m_{\infty},\ell^m_{1})}
$$
for any vectors $u_{1},\dots, u_{m}, v_{1},\dots,v_{m}$ in the closed unit ball of a given Hilbert space, $K_G$ being Grothendieck's constant. 
In particular, if $\{u_{1},\dots, u_{m}\}$ is chosen to be any orthonormal basis of $\ell_2^m$ and we take
$v_{k}:={\rm sign}(T_n(e_{i_k})(e_{j_{k}}))u_{k}$ for all $k\in\{1,\dots,m\}$, then the previous inequality, \eqref{eqn:IJ} and~\eqref{eqn:norm01} give
$$
	m\epsilon < \sum_{k=1}^m \big|T_n(e_{i_k})(e_{j_{k}})\big| \leq K_G \, m^{1/p+1/q}.
$$
Therefore $\varepsilon< K_G \, m^{1/p+1/q-1}$.
This inequality holds for all $m\in \N$, thus contradicting that $1/p+1/q<1$. The proof is finished.
\end{proof}

\begin{rem}\label{rem:HG}
\rm The proof of Theorem~\ref{theo:lp-pten-lq-sHG} shows that 
condition~(ii) of Theorem~\ref{theo:sHG} is satisfied without passing to a countable decomposition of 
$\{e_\gamma\otimes \tilde{e_{\gamma'}}: (\gamma,\gamma')\in \Gamma\times \Gamma\}$.
\end{rem}

\begin{cor}\label{cor:c0-pten-c0-sHG}
If $\Gamma$ and $\Delta$ are non-empty sets and $1<q<\infty$, then $c_0(\Gamma)\pten c_0(\Delta)$ and $c_0(\Gamma)\pten \ell_q(\Delta)$ 
are subspaces of Hilbert generated spaces.
\end{cor}
\begin{proof}
Take $1<p<\infty$ such that $1/p+1/q<1$. The formal inclusion
from $\ell_p(\Gamma)\pten \ell_q(\Delta)$ into $X:=c_0(\Gamma)\pten c_0(\Delta)$
(resp., $Y:=c_0(\Gamma)\pten \ell_q(\Delta)$) is a norm~$1$ operator with dense range, so its adjoint is injective and
gives a homeomorphic embedding of $(B_{X^*},w^*)$ (resp., $(B_{Y^*},w^*)$) into the uniform Eberlein compact
space $(B_{(\ell_p(\Gamma)\pten \ell_q(\Delta))^*},w^*)$. 
\end{proof}

The space $L_p(\mu)$, for a finite measure~$\mu$ and $1\leq p<\infty$, is subspace
of a Hilbert generated space (see Subsection~\ref{subsection:WCG}). Given any Banach space~$Y$,
the projective tensor product $L_1(\mu)\pten Y$ coincides with the Lebesgue-Bochner space $L_1(\mu,Y)$
(see, e.g., \cite[Section~2.3]{rya}), which is easily seen to be Hilbert generated (resp., subspace of a Hilbert generated space) whenever~$Y$ is.
The case $1<p<\infty$ is different:

\begin{rem}\label{rem:Lp-sHG}
\rm Let $1<p,q<\infty$ and let $\mu$ and $\nu$ be finite measures of uncountable Maharam type.
Then $L_p(\mu)\pten L_q(\nu)$ is not subspace of a Hilbert generated space.
Indeed, since the spaces $L_p(\mu)$ and $L_q(\nu)$ are non-separable, 
each contains a subspace isomorphic to~$\ell_2(\omega_1)$ 
(see, e.g., \cite[pp. 127-128, Theorems~9 and~12]{lac-J}), which can be seen to be complemented
like in the separable case (see, e.g., the proof of Theorem~4.53 in~\cite{fab-ultimo}).
Therefore, $L_p(\mu)\pten L_q(\nu)$ contains a complemented subspace isomorphic to $\ell_2(\omega_1)\pten \ell_2(\omega_1)$
(see, e.g., \cite[Proposition~2.4]{rya}). 
From Proposition~\ref{pro:embedding-lp-lq} it follows that $L_p(\mu)\pten L_q(\nu)$ 
contains a subspace isomorphic to~$\ell_1(\omega_1)$ and so it even fails property~(C).
\end{rem}

\section{Topological properties in injective tensor products}\label{section:WLDiten}

Some of the Banach space properties that we already considered are known to be stable by taking injective tensor products. 
Namely, given two Banach spaces $X$ and~$Y$, their injective tensor product $X\iten Y$
is WCG or subspace of a WCG space if (and only if) both $X$ and $Y$ have the corresponding property (see \cite[Section~2]{rue-wer}).
Indeed, the basic idea is to consider the natural isometric embedding 
$$
	X\iten Y
	\hookrightarrow C(K)
$$ 
where $K:=B_{X^*}\times B_{Y^*}$ is equipped with the product of the weak$^*$-topologies. Then the usual characterization
of relative weak compactness in $C(K)$ via pointwise convergence (see, e.g., \cite[Corollary~3.138]{fab-ultimo})
applies to conclude that 
{\em $W_1\otimes W_2$ is relatively weakly compact in $X\iten Y$ whenever $W_X \sub X$ and $W_Y \sub Y$ are relatively weakly compact.}
From this it follows at once that $X\iten Y$ is WCG whenever $X$ and $Y$ are WCG. If we only assume that $X$ and $Y$ are 
subspaces of WCG spaces, then $(B_{X^*},w^*)$ and $(B_{Y^*},w^*)$ are Eberlein and so the same holds for~$K$, hence
$C(K)$ is a WCG space (see, e.g., \cite[Theorem~14.9]{fab-ultimo}) containing $X\iten Y$ as a subspace.  A similar argument yields:

\begin{pro}\label{pro:sHG-iten}
Let $X$ and $Y$ be Banach spaces. Then $X\iten Y$ is subspace of a Hilbert generated space if and only if 
$X$ and $Y$ are subspaces of Hilbert generated spaces.
\end{pro}
\begin{proof}
Bear in mind that the product of two uniform Eberlein compact spaces is uniform Eberlein and that
$C(K)$ is Hilbert generated whenever $K$ is a uniform Eberlein compact space (see, e.g., \cite[Theorem~14.15]{fab-ultimo}).
\end{proof}

We next analyze Corson's property~(C) and the property of being WLD in injective tensor products, for which another approach is needed.

\begin{proposition}\label{prop:condineceinjewld}
Let $X$ and $Y$ be Banach spaces such that $X\iten Y$ has property~(C). The following statements hold: 
\begin{enumerate}
\item[(i)] The range of every element of $\mathcal{I}(X,Y^*)$ (resp., $\mathcal{I}(Y,X^*)$) is contained in a $w^*$-separable subset of~$Y^*$ (resp., $X^*$).
\item[(ii)] If $X$ has the $\lambda$-BSAP for some $\lambda\geq 1$, then $\mathcal{I}(X,Y^*) \sub \mathcal{S}(X,Y^*)$.  
\end{enumerate}
\end{proposition}

\begin{proof} We identify $(X\iten Y)^*$ with $\mathcal{I}(X,Y^*)$ (resp., $\mathcal{I}(Y,X^*)$) as in Subsection~\ref{subsection:iten}.

(i) Take $T\in \mathcal{I}(X,Y^*)$ with $\|T\|_{{\rm int}}=1$, so that
$$
	T\in B_{(X\iten Y)^*}=\overline{{\rm co}}^{w^*}\big(\{x^*\otimes y^*: \, (x^*,y^*)\in B_{X^*}\times B_{Y^*}\}\big).
$$
Since $X\iten Y$ has property~(C), there exist countable sets $A_1 \subseteq B_{X^*}$ and $A_2 \sub B_{Y^*}$ such that
$$
	T\in \overline{{\rm co}}^{w^*}\big(\{x^*\otimes y^*: \, (x^*,y^*)\in A_1\times A_2\}\big).
$$
For each $(x^*,y^*)\in A_1\times A_2$, the functional $x^*\otimes y^*$ is identified with the operator from $X$ to~$Y^*$
acting as $(x^*\otimes y^*)(x):=x^*(x) y^*$ for all $x\in X$. If $A$ denotes the set of such operators, then
$T$ belongs to the W$^*$OT-closure of ${\rm co}(A)$ in $\mathcal{L}(X,Y^*)$.
Clearly, this implies that $T(X)$ is contained in the $w^*$-separable set $\overline{{\rm span}}^{w^*} (A_2) \sub Y^*$.

(ii) is immediate from Theorem~\ref{theo:LEMMA}(b) applied to $V=B_{(X\iten Y)^*}$, 
bearing in mind that any Pietsch integral operator is weakly compact (see, e.g., \cite[Proposition~3.20]{rya}).
\end{proof}

\begin{theorem}\label{theo:caraWLDinject}
Let $X$ and $Y$ be Banach spaces. The following statements are equivalent:
\begin{enumerate}
\item[(i)] $X\iten Y$ is WLD.
\item[(ii)] $X$ and $Y$ are WLD, $\mathcal{I}(X,Y^*) \sub \mathcal{S}(X,Y^*)$ and $\mathcal{I}(Y,X^*) \sub \mathcal{S}(Y,X^*)$.
\end{enumerate}
\end{theorem}

\begin{proof}
Since the property of being WLD passes to subspaces and implies both property (C) (see Subsection~\ref{subsection:WLD})
and the $1$-BSAP (combine \cite[Theorem~3.42]{fab-alt-JJ} and Lemma~\ref{lem:SCPidentity}), it
follows from Proposition \ref{prop:condineceinjewld}(ii) that (i)$\impli$(ii). 

The converse follows the same lines of Theorem \ref{theorem:caractWLD}. Indeed, suppose that (ii) holds
let $\Phi$ be as in the proof of Theorem~\ref{theorem:caractWLD}. Since the formal inclusion 
$$
	j: (X\iten Y)^* \to \mathcal{B}(X,Y)=(X\pten Y)^*
$$ 
is an injective $w^*$-to-$w^*$ continuous operator, the composition 
$$
	\varphi:= \Phi \circ j:  (X\iten Y)^*  \to \ell_\infty(\Gamma)
$$	
is an injective $w^*$-to-$\tau_p(\Gamma)$ continuous operator.
Moreover, $\Phi(S) \in \ell_\infty^c(\Gamma)$ for every $S\in \mathcal{B}(X,Y)$ such that $S_X$ and $S_Y$ have separable range
(see the proof of Theorem~\ref{theorem:caractWLD}). Hence $\varphi$ takes
values in~$\ell_\infty^c(\Gamma)$ and so $X\iten Y$ is WLD.
\end{proof}

\begin{cor}\label{cor:WLDitenM}
Let $X$ and $Y$ be WLD Banach spaces. If either $(B_{X^*},w^*)$ or $(B_{Y^*},w^*)$ has property~(M), then $X\iten Y$ is WLD.
\end{cor}
\begin{proof} Suppose, for instance, that $(B_{X^*},w^*)$ has property~(M). On the one hand, since $L_1(\mu)$ is separable for every 
regular Borel probability measure on~$(B_{X^*},w^*)$ (see Subsection~\ref{subsection:Measures}), 
we can apply Pietsch's factorization theorem (see, e.g., \cite[2.13]{die-alt}) 
to deduce that every absolutely summing operator from~$X$ to another Banach space
has separable range. In particular, $\mathcal{I}(X,Y^*) \sub \mathcal{S}(X,Y^*)$. 

On the other hand, any $T\in \mathcal{I}(Y,X^*)$ factors as
$$
	\xymatrix@R=3pc@C=3pc{Y
	\ar[r]^{T} \ar[d]_{U} & X^*\\
	L_\infty(\mu)  \ar[r]^{I}  & L_1(\mu) \ar[u]_{V} \\
	}
$$
for some finite measure~$\mu$, where $I$ is the formal inclusion operator and $U$ and $V$ are operators. Since $X$ is WLD and
$(B_{X^*},w^*)$ has property~(M), we know that $V$ has norm separable range (Theorem~\ref{theo:JR}), and so does~$T$. 
Hence $\mathcal{I}(Y,X^*) \sub \mathcal{S}(Y,X^*)$. The conclusion now follows
from Theorem~\ref{theo:caraWLDinject}.
\end{proof}

\begin{theo}\label{theo:classMS}
If $X$ is a Banach space such that $(B_{X^*},w^*)$ does not belong to the class~MS, then $\mathcal{I}(X,X^*)\not\subseteq \mathcal{S}(X,X^*)$.
\end{theo}

\begin{proof}  Let $\mu$ be a regular Borel probability measure
on~$(B_{X^*},w^*)$ for which $L_1(\mu)$ is not separable.

{\em Step~1}. Define $S: X \to L_1(\mu)$ by $S := u \circ i$, where $i:X \to C(B_{X^*})$ is the canonical isometric embedding
and $u: C(B_{X^*}) \to L_1(\mu)$
is the formal inclusion operator. Let $A \sub C(B_{X^*})$
be the subalgebra generated by $i(B_X)\cup \{{\bf 1}\}$ (we denote by ${\bf 1}$ the constant function taking value~$1$). 
By the Stone-Weierstrass theorem
we have $C(B_{X^*})=\overline{A}^{\|\cdot\|}$ and so $u(A)$ is dense in~$L_1(\mu)$ (bear in mind that $u$ has dense range). Hence $u(A)$
is not separable. 

Since $u(f g)=u(f) u(g)$ for every $f,g\in C(B_{X^*})$, the set $u(A)$ consists of all linear combinations 
of finite products of elements of $H:=S(B_X)\cup \{{\bf 1}\}$. The following claim allows us to conclude that $S(B_X)$
is not separable.

{\em Claim:} For each $n\in \N$, the ``multiplication'' map
$$
	\xi_n: H^n \to L_1(\mu),
	\quad
	\xi_n(f_1,\dots,f_n):=f_1\cdots f_n,
$$
is $\tau$-to-norm continuous, where $\tau$ is the product topology on~$H^n$ induced by the norm topology on each factor. 
Indeed, we will show that $\xi_n(\overline{W}^\tau) \sub \overline{\xi_n(W)}^{\|\cdot\|}$ for every
$W \sub H^n$. Fix $(f_1,\dots,f_n) \in \overline{W}^\tau \sub H^n$ and take a sequence $((f_1^k,\dots,f_n^k))_{k\in \N}$
in~$W$ which $\tau$-converges to $(f_1,\dots,f_n)$. Then there is a strictly increasing sequence $k_1<k_2<\dots$ in~$\N$ such that
for each $i\in \{1,\dots,n\}$ the subsequence $(f_i^{k_j})_{j\in \N}$ is $\mu$-a.e. convergent to~$f_i$, so $(f_1^{k_j}\cdots f_n^{k_j})_{j\in \N}$
is $\mu$-a.e. convergent to $f_1\cdots f_n$. An appeal to Lebesgue's dominated convergence theorem
ensures that $(f_1^{k_j}\cdots f_n^{k_j})_{j\in \N}$
is norm convergent to $f_1\cdots f_n$ (bear in mind that $|f_i^k|\leq 1$ $\mu$-a.e. for all $i$ and~$k$). 
This shows that $\xi_n((f_1,\dots,f_n))\in \overline{\xi_n(W)}^{\|\cdot\|}$. The claim is proved.

{\em Step~2.} Let us consider the adjoint $S^*:L_\infty(\mu) \to X^*$. For each $w^*$-Borel set $C \sub B_{X^*}$ we have
\begin{equation}\label{eqn:barycenter-simple}
	\langle S^*(\chi_C), x \rangle=\int_C S(x) \, d\mu
	\quad
	\mbox{for all }x\in X,
\end{equation}
where $\chi_C$ denotes the characteristic function of~$C$.
Since $|S(x)|\leq 1$ $\mu$-a.e. for every $x\in B_X$, equality~\eqref{eqn:barycenter-simple}
can be used to deduce that 
$$
	\|S^*(\chi_C)\| \leq \mu(C)
	\quad\text{for every $w^*$-Borel set $C \sub B_{X^*}$}.
$$ 
This inequality and the density of simple functions in~$L_1(\mu)$ allow 
to define an operator $R: L_1(\mu) \to X^*$
such that $R(\chi_C)=S^*(\chi_C)$ for every $w^*$-Borel set $C\sub B_{X^*}$. From~\eqref{eqn:barycenter-simple} we get
\begin{equation}\label{eqn:barycenter}
	\langle R(f), x \rangle=\int_{B_{X^*}} f S(x) \, d\mu
	\quad
	\mbox{for all }x\in X \mbox{ and }f\in L_1(\mu).
\end{equation}
We will check that the operator $T:=R \circ S: X\to X^*$ satisfies the required properties. Clearly, $T$
is Pietsch integral because $S$ (and so~$T$) factors through the formal inclusion operator from $L_\infty(\mu)$ to~$L_1(\mu)$. 

Observe that $S$ factors as $S=J\circ v \circ i$, where 
$$
	v:C(B_{X^*}) \to L_2(\mu) \quad\mbox{and}\quad J:L_2(\mu) \to L_1(\mu)
$$
are the formal inclusion operators. Since $S$ has non-separable range (as we showed in {\em Step~1}), the same holds for $S':=v\circ i$. 
Then there exist $\epsilon>0$ and an uncountable set $\{x_\alpha:\alpha<\omega_1\}\sub B_X$ such that 
\begin{equation}\label{eqn:separation}
	\|S'(x_\alpha-x_\beta)\|=\|S'(x_\alpha)-S'(x_\beta)\| \geq \epsilon
	\quad\mbox{whenever }\alpha \neq \beta.
\end{equation}
It follows that
\begin{multline*}
	2 \|T(x_\alpha)-T(x_\beta)\| \geq \big\langle T(x_\alpha)-T(x_\beta), x_\alpha-x_\beta \big\rangle= \big\langle R(S(x_\alpha-x_\beta)),x_\alpha-x_\beta\big\rangle
	\\ \stackrel{\eqref{eqn:barycenter}}{=}\int_{B_{X^*}} S(x_\alpha-x_\beta)^2 \, d\mu=
	\|S'(x_\alpha-x_\beta)\|^2 \stackrel{\eqref{eqn:separation}}{\geq} \epsilon^2
\end{multline*}
whenever $\alpha\neq \beta$. Therefore, $T$ has non-separable range.
\end{proof}

Recall that, if $X$ is a WLD Banach space, then $(B_{X^*},w^*)$ has property~(M)
if and only if it belongs to the class~MS (see Subsection~\ref{subsection:Measures}). Thus:

\begin{cor}\label{cor:WLDiten-noM}
Let $X$ be a Banach space. Then $X$ is WLD and $(B_{X^*},w^*)$ has property~(M) if and only if $X\iten X$ is WLD.
\end{cor}
\begin{proof}
Apply Corollary~\ref{cor:WLDitenM} and Theorems~\ref{theo:caraWLDinject} and~\ref{theo:classMS}. 
\end{proof}

\begin{cor}\label{cor:PropertyCiten-MS}
Let $X$ be a Banach space having the $\lambda$-BSAP for some $\lambda\geq 1$ such that $X\iten X$ has property~(C). 
Then $(B_{X^*},w^*)$ belongs to the class~MS.  
\end{cor}
\begin{proof}
Apply Proposition~\ref{prop:condineceinjewld}(ii) and Theorem~\ref{theo:classMS}.
\end{proof}

Since $C(K)$ has the $1$-BAP for every compact space~$K$ 
and $C(K\times K)$ is isometrically isomorphic to $C(K)\iten C(K)$ (see, e.g., \cite[Sections~4.1 and~3.2]{rya}),
from the previous corollary we get the following result (see \cite[Theorem~5.6]{ple-sob}): 

\begin{cor}[Plebanek-Sobota]\label{cor:PlebanekSobota}
Let $K$ be a compact space. If $C(K \times K)$ has property~(C), then $K$ belongs to the class~MS.
\end{cor}

It should be mentioned that a compact space~$K$ belongs to the class~MS if and only if 
$(B_{C(K)^*},w^*)$ belongs to the class~MS, see \cite[Proposition~2.4]{mar-ple-2}.

\begin{rem}\label{rem:CH}
\rm As we already mentioned in Subsection~\ref{subsection:Measures}, under CH there exist WLD Banach spaces $X$ for which $(B_{X^*},w^*)$ 
fails property~(M) (equivalently, it does not belong to the class~MS) and so $X\iten X$ does not have property~(C) by Corollary \ref{cor:PropertyCiten-MS}.
Since $(B_{X^*},w^*)$ is angelic (i.e.,  Fr\'{e}chet-Urysohn) whenever $X$ is a WLD Banach space,
this provides a (consistent) negative answer to the question raised in \cite[Problem~2.7]{rue-wer} of whether the property
of having $w^*$-angelic dual ball is preserved by the injective tensor product. 
\end{rem}

\section{Questions}\label{section:Questions}

We finish the paper with some related questions that we have been 
unable to answer. In what follows, 
$X$ and $Y$ are Banach spaces.

\begin{enumerate}
\item[(a)] Suppose that $X$ and $Y$ have property~(C) and that
\begin{equation}\label{eqn:A}
	\mathcal{L}(X,Y^*)=\mathcal{K}(X,Y^*).
\end{equation}
Does $X\pten Y$ have property~(C)? What happens if~\eqref{eqn:A} is weakened to
$$ 
	\mathcal{L}(X,Y^*)= \mathcal{S}(X,Y^*) \quad\mbox{and}\quad \mathcal{L}(Y,X^*)= \mathcal{S}(Y,X^*)?
$$
What about the particular case when either $X$ or~$Y$ is reflexive?
\item[(b)] Suppose that $X$ and $Y$ are WLD and that either $X$ or $Y$ has the Dunford-Pettis property. Is $X\pten Y$ WLD?
\item[(c)] Suppose that $X$ and $Y$ are subspaces of Hilbert generated spaces and that equality~\eqref{eqn:A} holds.
Is $X \pten Y$ subspace of a Hilbert generated space?
\item[(d)] Suppose that $X\iten Y$ is WLD. Does either $(B_{X^*},w^*)$ or $(B_{Y^*},w^*)$ have property~(M)?
\end{enumerate}

\subsection*{Acknowledgements}

The research is partially supported by grants MTM2017-86182-P 
(funded by MCIN/AEI/10.13039/501100011033 and ``ERDF A way of making Europe'') and 
20797/PI/18 (funded by {\em Fundaci\'on S\'eneca}).
A. Rueda Zoca was also partially supported by: grant PGC2018-093794-B-I00 
(funded by MCIN/AEI/10.13039/501100011033 and ``ERDF A way of making Europe''), 
grant FJC2019-039973 (funded by MCIN/AEI/10.13039/501100011033) and grants A-FQM-484-UGR18 and FQM-0185 (funded by {\em Junta de Andaluc\'{i}a}).

%\bibliography{C:/Mat/Bibliografia/inves,C:/Mat/Bibliografia/invesJose}

\bibliographystyle{amsplain}

\end{document}